\documentclass{amsart}

\usepackage{amsmath}\usepackage{amscd}\usepackage{amsfonts}\usepackage{amsthm}\usepackage{amssymb}\usepackage{enumerate}\usepackage{graphicx}
\graphicspath{{figures/}}\usepackage[all]{xy}\usepackage{color}

\numberwithin{equation}{subsection}

\theoremstyle{plain}
\newtheorem{theorem}{Theorem}[section]
\newtheorem{lemma}[theorem]{Lemma}

\newtheorem{proposition}[theorem]{Proposition}
\theoremstyle{definition}

\theoremstyle{remark}
\newtheorem{remark}[theorem]{Remark}
\newtheorem*{acknowledgments}{Acknowledgments}

\newcommand\adm{\mathrm{adm}}

\newcommand\Lapq{\modL ^{\adm}_{M;p,q}}
\newcommand\GL{\mathrm{GL}}
\newcommand\Opq{\mathrm{O}(p,q;\Z)}

\newcommand\modL {\mathcal{L}}
\newcommand\modP {\mathcal{P}}
\newcommand\modQ {\mathcal{Q}}
\newcommand\modW {\mathcal{W}}
\newcommand\modD {\mathcal{D}}
\newcommand\modC {{\mathcal C}}
\newcommand\modS {{\mathcal S}}
\newcommand\modE {{\mathcal E}}

\newcommand\Q {{\mathbb Q}}
\newcommand\Z{{\mathbb Z}}
\newcommand\xto[1]{{\overset{#1}{\longrightarrow}}}
\newcommand\xtoc[1]{{\overset{#1}{\underset{\cong}{\longrightarrow}}}}
\newcommand\id{\operatorname{id}}
\newcommand\co {\colon\thinspace}
\newcommand\FI[2]{
  \begin{figure}[t]
    \begin{center}
      \input{#1.pdf_tex}
    \end{center}
    \caption{#2}
    \label{#1}
\end{figure}}
\newcommand\FIG[1]{
  \begin{center}
    \input{#1.pdf_tex}
  \end{center}}

\newcommand\opint{\operatorname{int}}
\newcommand\ct{\overset{\cong}{\to}}
\newcommand\bq{\bar{q}}
\newcommand\tx{\tilde{x}}
\newcommand\tL{\tilde{L}}
\newcommand\tl{\tilde{l}}

\newcommand\onto{\twoheadrightarrow}

\newcommand\incl{\textrm{incl}}

\newcommand\Ob{\operatorname{Ob}}
\newcommand\tr{c}

\newcommand\rank{\operatorname{rank}}
\newcommand\End{\operatorname{End}}
\newcommand\tri[1]{(#1,\rho _{#1},\phi _{#1})}
\newcommand\triM{\tri{M}}
\newcommand\triW{\tri{W}}
\newcommand\pt{\textrm{pt}}
\newcommand\Lk{\operatorname{Lk}}

\begin{document}
\title{Kirby calculus for null-homologous framed links in $3$--manifolds}

\date{April 6, 2017 (First version: April 1, 2014)}

\author{Kazuo Habiro}

\author{Tamara Widmer}

\address{Research Institute for Mathematical Sciences, Kyoto
  University, Kyoto 606-8502, Japan}

\address{Universit\"at Z\"urich, Winterthurerstr. 190 CH-8057
  Z\"urich, Switzerland}

\begin{abstract}
  A theorem of Kirby gives a necessary and sufficient condition for two
  framed links in $S^3$ to yield orientation-preserving diffeomorphic
  results of surgery.  Kirby's theorem is an important method for
  constructing invariants of $3$--manifolds.  In this paper, we prove a
  variant of Kirby's theorem for null-homologous framed links in a
  $3$--manifold.  This result involves a new kind of moves, called
  IHX-moves, which are closely related to the IHX relation in the theory
  of finite type invariants.  When the first homology group of $M$ is
  free abelian, we give a refinement of this result to $\pm1$--framed,
  algebraically split, null-homologous framed links in~$M$.
\end{abstract}

\maketitle

\section{Introduction}
\label{sec:introduction}

\subsection{Kirby calculus for framed links in $3$--manifolds}

Surgery along a framed link $L$ in a $3$--manifold $M$ is a process
of removing a tubular neighborhood of $L$ from $M$ and gluing back a
solid torus in a different way using the framing, which yields a new
$3$--manifold $M_L$.  One can also construct $M_L$ by
using the $4$--manifold $W_L$ obtained from the cylinder
$M\times[0,1]$ by attaching $2$--handles on $M\times \{1\}$ along
$L\times \{1\}$.  Then $W_L$ is a cobordism between $M_L$ and $M$.

Every closed, connected, oriented $3$--manifold can be obtained from
the $3$--sphere $S^3$ by surgery along a framed link
\cite{Lickorish,Wallace}.  Kirby's calculus of framed links \cite{Kirby} gives a
criterion for two framed links in $S^3$ to produce
orientation-preserving diffeomorphic result of surgery: two framed
links $L$ and $L'$ in $S^3$ yield orientation-preserving diffeomorphic
$3$--manifolds if and only if $L$ and $L'$ are related by a sequence of
two kinds of moves, called {\em stabilizations} and {\em
handle-slides}, depicted in Figure \ref{K1-K2}.
\FI{K1-K2}{(a) A stabilization (or $K_1$--move) adds or deletes an
isolated $\pm 1$--framed unknot.  (b) A handle-slide (or $K_2$--move)
replaces one component with the band-sum of the component with a
parallel copy of another component.}
These moves are also called $K_1$--moves and $K_2$--moves in the literature.
Kirby's theorem is used in the definition of the Reshetikhin--Turaev
invariant \cite{RT} and other quantum $3$--manifold invariants.

Fenn and Rourke \cite{FR} generalized Kirby's theorem to framed links
in a general closed $3$--manifold in two natural ways.
On the one hand, they proved that two framed links in $M$ yield
orientation-preserving diffeomorphic $3$--manifolds if and only if they
are related by a sequence of stabilizations, handle-slides and {\em
$K_3$--moves}.  Here a $K_3$--move on a framed link adds or removes a
$2$--component sublink $K\cup K'$ such that $K$ is a framed knot in $M$
with arbitrary framing, and $K'$ is a small $0$--framed knot meridional
to $K$, see Figure \ref{K3move}.  Roberts \cite{Roberts}
generalized this result to $3$--manifolds with boundary.
\FI{K3move}{A $K_3$--move $L\leftrightarrow L\cup K\cup K'$.}
On the other hand, Fenn and Rourke considered the equivalence
relation, called the {\em $\delta $--equivalence}, on framed links in $M$
generated by stabilizations and handle-slides.  They proved that two
framed links $L$ and $L'$ in a closed oriented $3$--manifold $M$ are
$\delta $--equivalent if and only if $\pi _1W_L$ and $\pi _1W_{L'}$ are
isomorphic and there is an orientation-preserving diffeomorphism
$h\co M_L\ct M_L'$ satisfying a certain condition.  This result is
generalized to $3$--manifolds with boundary \cite{HW}.  (See also
\cite{GK} for the case where the boundary is connected.)

\subsection{Kirby calculus for null-homologous framed links}

The main purpose of this paper is to study {\em calculus of null-homologous
framed links} in a compact, connected, oriented $3$--manifold $M$
possibly with non-empty boundary.

Let $k$ be $\Z  $ or $\Q$.  A framed link $L$ in $M$ is said to be {\em
$k$--null-homologous} if every component of $L$ is $k$--null-homologous
in $M$, i.e., represents $0\in H_1(M;k)$.

Let $P\subset \partial M$ be a subset which contains exactly one point of each
connected component of $\partial M$.  To a $k$--null-homologous framed link
$L$ in $M$ is associated a surjective homomorphism
\begin{gather}
  \label{e32}
  g_L\co H_1(M_L,P_L;k)\rightarrow H_1(M,P;k)
\end{gather}
defined as the composite
\begin{gather*}
  H_1(M_L,P_L;k )
  \overset{\incl_*}{\longrightarrow}
  H_1(W_L,P_L;k)
  \overset{}{\underset{\cong}{\longrightarrow}}
  H_1(W_L,P;k)
  \overset{\incl_*^{-1}}{\underset{\cong}{\longrightarrow}}
  H_1(M,P;k ).
\end{gather*}
Here $P_L\subset \partial M_L$ is the image of $P$ by the natural identification
map $\partial M\ct\partial M_L$, and the middle isomorphism is induced
by the paths $p\times[0,1]\subset\partial M\times[0,1]\subset\partial
W_L$ for $p\in P$.

A {\em $k$--null-homologous $K_3$--move} is a $K_3$--move such that the
component $K$ in the definition of a $K_3$--move is
$k$--null-homologous.  A $k$--null-homologous $K_3$--move
transforms a $k$--null-homologous framed link into another
$k$--null-homologous framed link.

We will define an {\em IHX-move} in Section \ref{sec6}.  This
move corresponds to the IHX relation for tree claspers, and is closely
related to the IHX relation in the theory of finite type invariants of
links and $3$--manifolds.

The first main result in this paper is the following.

\begin{theorem}
  \label{Qnull-homologous}
  Let $M$ be a compact, connected, oriented $3$--manifold.  Let
  $P\subset \partial M$ be a subset containing exactly one point of
  each connected component of $\partial M$.  Let $L$ and $L'$ be
  $\Q$--null-homologous framed links in $M$.  Then the following
  conditions are equivalent.
  \begin{enumerate}
  \item $L$ and $L'$ are related by a sequence of stabilizations,
  handle-slides, $\Q$--null-homologous $K_3$--moves, and IHX-moves.
  \item There is an orientation-preserving diffeomorphism $h\co
    M_L\ct M_{L'}$ restricting to the
    identification map $\partial M_{L}\cong \partial M_{L'}$ such that the following
    diagram commutes.
    \begin{gather}
      \label{e29}
      \xymatrix{
	H_1(M_L,P_L;\Q ) \ar[rr]^{h_*}_{\cong}\ar[rd]_{g_L}
	&&
	H_1(M_{L'},P_{L'};\Q )\ar[dl]^{g_{L'}}\\
	&H_1(M,P;\Q ).
      }
    \end{gather}
  \end{enumerate}
\end{theorem}

In Theorem \ref{Qnull-homologous}, IHX-moves are necessary only when
$\rank H_1M\ge4$, see Remark \ref{r3}.  If $M$ is a rational homology
sphere, then Theorem \ref{Qnull-homologous} without IHX-moves recovers
\cite[Theorem 8]{FR}, since every framed link in $M$ is
$\Q$--null-homologous.

Theorem \ref{Qnull-homologous} is proved in Section
\ref{sec:null-homol-fram} for $M$ with non-empty boundary, and in
Section \ref{sec:proof-theor-refqn} for $M$ closed.

If $H_1(M;\Z  )$ is free abelian, then $\Q$--null-homologous framed links
in $M$ are $\Z  $--null-homologous.  It is easy to see that Theorem
\ref{Qnull-homologous} implies the following.

\begin{theorem}
  \label{Znull-homologous}
  Let $M$ and $P$ be as in Theorem \ref{Qnull-homologous}, and assume
  that $H_1(M;\Z  )$ is free abelian.  Let $L$ and $L'$ be $\Z  $--null-homologous
  framed links in $M$.  Then the following conditions are equivalent.
  \begin{enumerate}
  \item $L$ and $L'$ are related by a sequence of stabilizations,
    handle-slides, $\Z  $--null-homologous $K_3$--moves, and IHX-moves.
  \item There is an orientation-preserving diffeomorphism $h\co
    M_L\ct M_{L'}$ restricting to the
    identification map $\partial M_{L}\cong \partial M_{L'}$ such that the following diagram
    commutes.
    \begin{gather}
      \label{e30}
      \xymatrix{
	H_1(M_L,P_L;\Z  ) \ar[rr]^{h_*}_{\cong}\ar[rd]_{g_L}
	&&
	H_1(M_{L'},P_{L'};\Z  )\ar[dl]^{g_{L'}}\\
	&H_1(M,P;\Z  ).
      }
    \end{gather}
  \end{enumerate}
\end{theorem}

In a future paper \cite{H2}, the first author plans to generalize
Theorem \ref{Znull-homologous} to the case where $H_1(M;\Z )$ has
non-trivial torsion.

\subsection{Kirby calculus for admissible framed links}

The second main result of this paper is a refinement of Theorem
\ref{Znull-homologous} to a special class of framed links, called {\em
admissible framed links}.

Let $L=L_1\cup \dots \cup L_k$ be a $\Z  $--null-homologous framed link in a compact, connected,
oriented $3$--manifold $M$.
Note that $L$ admits a well-defined {\em linking matrix}
\begin{gather*}
  \Lk(L)=(l_{ij})_{1\le i,j\le k},
\end{gather*}
where $l_{ii}$ is the framing of $L_i$, and $l_{ij}$ ($i\neq j$) is
the linking number of $L_i$ and $L_j$.  Then $L$ is said to be {\em
admissible} if the linking matrix $\Lk(L)$ is diagonal with diagonal
entries $\pm 1$.  We call surgery along an admissible framed link an
{\em admissible surgery}.

Admissible surgeries on $3$--manifolds have been studied in several
places.  It is well known that every integral homology $3$--sphere can
be obtained from $S^3$ by an admissible surgery.  Ohtsuki used
admissible surgeries to define the notion of finite type invariants of
integral homology spheres \cite{Ohtsuki:FTI}.
Cochran, Gerges and Orr \cite{CGO} studied the equivalence relation on
closed, oriented $3$--manifolds generated by admissible surgeries,
called {\em $2$--surgeries} in \cite{CGO}.  They gave a
characterization for two closed oriented $3$--manifolds to be
equivalent under admissible surgeries.
Cochran and Melvin \cite{CM} used admissible surgeries to define
finite type invariants of $3$--manifolds generalizing Ohtsuki's
definition of finite type invariants of integral homology spheres.

A {\em band-slide} on a framed link is an algebraically cancelling
pair of handle-slides, see Figure~\ref{bandslide}.
\FI{bandslide}{A band-slide of the component $L_i$ over
$L_j$.}
A band-slide
preserves the homology classes of the components of a link, and also
preserves the linking matrix of a $\Z $--null-homologous framed link.
Thus, a band-slide on an admissible framed link yields an admissible
framed link.

The first author proved the following refinement of Kirby's theorem to
admissible framed links in $S^3$.

\begin{theorem}[\cite{H:kirby}]
  \label{r45}
  Let $L$ and $L'$ be two admissible framed links in $S^3$.  Then the
  following conditions are equivalent.
  \begin{enumerate}
  \item $L$ and $L'$ are related by a sequence of stabilizations
  and band-slides.
  \item $S^3_L$ and $S^3_{L'}$ are orientation-preserving
  diffeomorphic.
  \end{enumerate}
\end{theorem}

The second main result of this paper is Theorem \ref{admissible}
below, which refines Theorem \ref{Znull-homologous} and generalizes
Theorem \ref{r45}.  To state it, we need new moves on admissible
framed links called {\em pair
  moves}, {\em lantern-moves} and {\em admissible IHX-moves}.

Two admissible framed links in $M$ are said to be related by a {\em
pair move} if one of them, say $L'$, is obtained from the other, say
$L$, by adjoining a $2$--component admissible framed link $K^+\cup
K^-$ in $M\setminus L$, where $K^+$ and $K^-$ are parallel to each
other and $K^+$ and $K^-$ have framings $+1$ and $-1$, respectively,
see Figure \ref{pair-move}.  It follows that $L$ and $L'$ give
diffeomorphic results of surgery, since one can handle-slide $K^+$
over $K^-$ to obtain from 
$L'=L\cup K^+\cup K^-$ a framed link $\tilde{L}=L\cup J\cup K^-$ which is related to $L$
by a $\Z$--null-homologous $K_3$--move.  \FI{pair-move}{A pair-move}

A {\em lantern-move} is defined as follows.  Let $V_3$ be a handlebody
of genus $3$, which is identified with the complement of the tubular
neighborhood of a trivial $3$--component string link
\begin{gather*}
  \gamma =\gamma _1\cup \gamma _2\cup \gamma _3=(\text{three points})\times [0,1]
\end{gather*} in the cylinder $D^2\times [0,1]$.  Let $K$ and $K'$ be
two framed links in $V_3$ as depicted in Figure \ref{lantern}(a) and
(b), respectively.  \FI{lantern}{A lantern-move in a genus $3$
handlebody $V_3$.}  Here all the components in $K$ and $K'$ are $+1$,
where the framings are defined in the cylinder.  The framed links $K$
and $K'$ correspond to the products of Dehn twists in a $3$--punctured
torus $D_3$ appearing in the two sides of the lantern relation in the
mapping class group.  Let $L$ be an admissible framed link in a
$3$--manifold $M$, and let $f\co V_3\hookrightarrow M\setminus L$ be
an orientation-preserving embedding such that both $L\cup f(K)$ and
$L\cup f(K')$ are admissible in $M$.  (In fact, $L\cup f(K)$ is
admissible if and only if $L\cup f(K')$ is admissible.)  Then the two
framed links $L\cup f(K)$ and $L\cup f(K')$ are said to be related by
a lantern-move.  A lantern-move preserves the diffeomorphism class of
the results of surgery since we have a diffeomorphism
$(V_3)_{K}\cong(V_3)_{K'}$ restricting to the canonical map $\partial
(V_3)_{K}\cong\partial (V_3)_{K'}$.  The latter fact follows since the
results of surgery along $K$ and $K'$ on the framed string link $\gamma \subset D^2\times [0,1]$
 are equivalent.  Alternatively, one
can check that $K$ and $K'$ are $\delta $--equivalent in $V_3$.

An {\em admissible IHX-move}, defined in Section
\ref{sec:admissible-ihx-moves-1}, is a modification of an IHX-move,
involving only admissible framed links.

\begin{theorem}
  \label{admissible}
  Let $M$, $P$ be as in Theorem \ref{Znull-homologous}.  Let $L$ and
  $L'$ be admissible framed links in $M$.  Then the following
  conditions are equivalent.
  \begin{enumerate}
  \item $L$ and $L'$ are related by a sequence of stabilizations,
    band-slides, pair-moves, admissible IHX-moves, and lantern-moves.
  \item There is an orientation-preserving diffeomorphism $h\co
    M_L\ct M_{L'}$ restricting to the natural
    identification map $\partial M_{L}\ct \partial  M_{L'}$ such that Diagram
    \eqref{e30} commutes.
  \end{enumerate}
\end{theorem}

The set of moves given in Theorem \ref{admissible}(ii) is not minimal.
See Section \ref{sec:realizing-moves-with}.  In particular, admissible
IHX-moves are not necessary when $\rank H_1M<4$.  If, moreover, $M$ is
$S^3$, then we do not need the pair-moves or the lantern-moves and we
recover Theorem \ref{r45}.

Theorem \ref{admissible} gives a method to study $3$--manifolds
obtained by admissible surgery on a fixed $3$--manifold $M$ with free
abelian $H_1M$.  Thus, it is expected that Theorem \ref{admissible}
has some applications to the surgery equivalence relation studied by
Cochran, Gerges and Orr \cite{CGO} and to Cochran and Melvin's finite
type invariants \cite{CM}.

In particular, in a future paper the first author plans to use Theorem
\ref{admissible} in order to construct
an invariant $J_2(M)$ of $3$--manifolds $M$ which are admissible
surgery equivalent to a fixed $3$--manifold $M_0$ with
$H_1M_0\cong\Z^r$, $r\ge0$.  If $M_0=\Sigma_{g,1}\times[0,1]$ is a
cylinder over a surface $\Sigma_{g,1}$ of genus $g$ with
$\partial\Sigma_{g,1}\cong S^1$, the manifolds we consider here are
exactly homology cylinders over $\Sigma_{g,1}$ with vanishing first
Johnson homomorphism, and the invariant $J_2(M)$ coincides with the
image of $M$ by the second Johnson homomorphism.  Thus, the invariant
$J_2(M)$ may be regarded as a natural generalization of the second
Johnson homomorphism of homology cylinders.

\subsection{Organization of the paper}
\label{sec:organization-paper}

The rest of this paper is organized as follows.
In Section \ref{sec2} we recall the generalization of Fenn and
Rourke's theorem to $3$--manifolds with boundary obtained in \cite{HW}.
We use fundamental groupoids of $3$--manifolds to simplify the statement.
This result is modified in Section \ref{sec3} for $N$--links in a
$3$--manifold $M$.  Here $N$ is a normal subgroup of $\pi_1M$, and an
$N$--link is a framed link such that the homotopy class of each
component is contained in $N$.
In Section \ref{sec4} we fix a group $G$ and consider manifolds over
the Eilenberg--Mac Lane space $K(G,1)$.  This section includes further
modification of the result in Section \ref{sec3} and preparation for
the next section.
In Section \ref{sec5} we show that each homology class in $H_4(G)$
can be realized by a framed link in a handlebody $V$.
Sections  \ref{sec6} and \ref{sec7} are preparations for Section
\ref{sec8}.
In Section \ref{sec6} we define IHX-moves on null-homologous framed
links in a $3$--manifold, and in Section \ref{sec7} we give a new
handle decomposition of the $4$--torus $T^4$ related to the IHX-move.
In Sections \ref{sec8} and \ref{sec9} we prove Theorems
\ref{Qnull-homologous} and \ref{admissible}, respectively.

\begin{acknowledgments}
  The first author was partially supported by JSPS, Grant-in-Aid for
  Scientific Research (C) 24540077.  The second author was supported
  by SNF, No.\ 200020\_134774/1.  The first author thanks Christian
  Blanchet and Gregor Masbaum for helpful discussions.
\end{acknowledgments}

\section{Generalization of Fenn and Rourke's theorem to $3$--manifolds
 with boundary}
\label{sec2}

In this section, we state the generalization of Fenn and Rourke's
theorem \cite{FR} to $3$--manifolds with boundary that we proved in
\cite{HW}.  We mainly follow the constructions in \cite{HW}, but the
description given here is slightly simplified by the use of fundamental
groupoids.

\subsection{Fundamental groupoids}
\label{sec:fund-group}
Let $X$ be a topological space, and let $P\subset X$ be a subset.  Let
$\Pi (X,P)$ denote the fundamental groupoid of $X$ with respect to
$P$.  The objects of $\Pi (X,P)$ are the elements of $P$, and the
morphisms from $p\in P$ to $p'\in P$ are homotopy classes of paths
from $p$ to $p'$.  The set $\Pi (X,P)(p,p')$ of morphisms from $p$ to
$p'$ is denoted usually by $\Pi (X;p,p')$.  For $p\in P$, we set
$\pi _1(X,p)=\Pi (X;p,p)$, the fundamental group of $X$ at $p$.

If $X$ is connected, then $\Pi (X,P)$ is a connected groupoid, i.e., for
$p,p'\in P$, the set $\Pi (X;p,p')$ is non-empty.  If, moreover, $N$ is a
normal subgroup of $\pi _1(X,p)$, then we denote by $\Pi (X,P)/N$ the
quotient of $\Pi (X,P)$ by the equivalence relation $\sim$ on
the morphisms in $\Pi(X,P)$ such that for $g,g'\co p'\to p''$ ($p',p''\in P$), we have
$g\sim g'$ if and only if for any $f\co p\to p'$ we have $f^{-1}g^{-1}g'f\in N$.

In this paper, $\twoheadrightarrow$ for groupoids denotes an
epimorphism in the category of groupoids, i.e., a full functor which
is surjective on objects.  In fact, all the groupoid epimorphism
denoted by
$\twoheadrightarrow$ appearing below are be bijective on objects.  In the above
situation we have an epimorphism
\begin{gather*}
  \Pi (X,P)\twoheadrightarrow\Pi (X,P)/N,
\end{gather*}
which is the identity on objects.

In this section, we fix a compact, connected, oriented $3$--manifold
$M$ with non-empty boundary, whose components will be denoted by
$F_1, \ldots, F_t$ ($t\ge1$).  We also fix
\begin{gather*}
  P=\{p_1,\ldots ,p_t\}\subset \partial M,
\end{gather*}
where $p_i\in F_i$ for each $i=1,\ldots ,t$.
We consider the fundamental groupoid $\Pi (M,P)$ of $M$ with respect to
$P$.  Since $M$ is connected, the groups $\pi _1(M,p_i)$ for $i=1,\ldots ,t$
are isomorphic to each other.  We regard $p_1$ as the basepoint of
$M$, and often write $\pi _1M=\pi _1(M,p_1)$.

\subsection{Framed links and surgery}
\label{sec:framed-links-surgery}

A {\em framed link} $L=L_1\cup\dots\cup L_n$ in $M$ is a link such
that each component $L_i$ of $L$ is given a framing, i.e., a homotopy
class of a simple closed curve $\gamma_i$ in the boundary $\partial
N(L_i)$ of a tubular neighborhood $N(L_i)$ of $L_i$ in $M$ which is
homotopic to $L_i$ in $N(L_i)$. {\em Surgery} along a framed link $L$
denotes the process of removing the interior of $N(L_i)$, and gluing
a solid torus $D^2\times S^1$ to $\partial N(L_i)$ so that the curve
$\partial D^2\times \{*\}$, $*\in S^1$, is attached to
$\gamma_i\subset\partial N(L_i)$ for $i=1,\ldots,n$.  We denote
the result of surgery by $M_L$.  Note that the boundary $\partial M_L$ is
naturally identified with $\partial M$.

Surgery along a framed link can also be defined by using
$4$--manifolds, as mentioned in the introduction.  In the above
situation, let $W_L$ denote the
$4$--manifold obtained from the cylinder $M \times I$, where $I=[0,1]$,
by attaching a $2$--handle $D^2\times D^2$ along $N(L_i)\times\{1\}$
using a diffeomorphism
\begin{gather*}
  S^1\times D^2\overset{\cong}{\to} N(L_i),
\end{gather*}
which maps $S^1\times\{*\}$, $*\in\partial D^2$, onto the framing $\gamma_i$.
We have a natural identification
$$\partial W_L\cong M\underset{\partial M}{\cup}
(\partial M\times I)\underset{\partial M_L}{\cup} M_L.$$
Thus, $W_L$ is a cobordism between $M$ and $M_L$.
Note that $\partial W_L$ is a connected, closed $3$--manifold.

Set $P_L=\{p_1^L,\dots,p_t^L\}\subset \partial M_L$, where
$p_k^L=p_k\times\{1\}\in\partial M_L$ for $k=1,\dots,t$.  Let
$\gamma_k=p_k\times I \subset \partial W_L$ for $k=1,\dots,t$.  Note
that $\gamma _k$ is an arc in $\partial W_L$ from $p_k\in\partial
M\subset \partial W_L$ to $p_k^L$.

The point $p_1$ is regarded as a basepoint of $W_L$ as well as of
$M$, and we set $\pi _1W_L:=\pi _1(W_L,p_1)$.  We regard $p_1^L$ as the
basepoint of $M_L$ and write $\pi _1M_L:=\pi _1(M,p_1^L)$.

The inclusions
\begin{equation*}
  M \overset{i}{\hookrightarrow} W_L \overset{i'}{\hookleftarrow} M_L
\end{equation*}
induce full functors
\begin{gather*}
  \xymatrix{
    \Pi (M,P)\ar@{->>}[r]^{i_*}& \Pi (W_L,P)
    & \Pi (M_L,P_L)   . \ar@{->>}[l]_{i'_*}
  }
\end{gather*}
Here $i'_*$ is defined as the composite
\begin{gather*}
  \Pi (M_L,P_L)\overset{i'_*}{\longrightarrow} \Pi (W_L,P_L)\underset{\cong}{\overset{\gamma_1,\dots,\gamma_t}{\longrightarrow}} \Pi (W_L,P),
\end{gather*}
where the second isomorphism is induced by the arcs $\gamma_1,\dots,\gamma_t$.

Let $N_L$ denote the normal subgroup of $\pi _1M$ normally
generated by the homotopy classes of the components of $L$.  Then we have
\begin{gather}
  \label{e11}
  N_L=\ker(i_*\co \pi _1M\rightarrow \pi _1W_L).
\end{gather}

\subsection{Fenn--Rourke theorem for $3$--manifolds with boundary}
\label{sec:constr-homol-class}

Fenn and Rourke \cite[Theorem 6]{FR} characterized the condition for
two framed links in a closed, oriented $3$--manifold to be related by
a sequence of stabilizations and handle-slides.  Garoufalidis and
Kricker \cite{GK} and the authors \cite{HW} generalized it to
$3$--manifolds with boundary.
(The result in \cite{GK} was not correct for $3$-manifold with more
than one boundary components.  A correct version was given in
\cite{HW}.)
In this subsection we state this result in a slightly different way
using fundamental groupoids.

Let $L$ and $L'$ be framed links in $M$ and suppose that there is an
orientation-preserving diffeomorphism
\begin{gather*}
  h\co M_L\overset{\cong}{\rightarrow} M_{L'}
\end{gather*}
which restricts to the natural identification map
$\partial M_L\cong\partial M_{L'}$.
Then we obtain a closed, oriented $4$--manifold
\begin{gather*}
  W=W_{M,L,L',h}:= W_L\cup _\partial (-W_{L'})
\end{gather*}
by gluing $W_L$ with $-W_{L'}$ along their boundaries using the map
\begin{gather*}
 h\cup \id_{(M\times \{0\})\cup (\partial M\times [0,1])} \co \partial W_L\ct\partial W_{L'}.
\end{gather*}

Suppose that we have $N_L=N_{L'}$.  Then there exists a unique
groupoid isomorphism $f\co  \Pi (W_{L},P) \to \Pi (W_{L'},P)$, which is the
identity on objects, such that the triangle in the diagram
\begin{gather} \label{d1}
  \xymatrix{
  \Pi (M_{L},P_L) \ar[rr]^{h_*}_{\cong}\ar@{->>}[d]_{i'_*} & & \Pi (M_{L'},P_{L'}) \ar@{->>}[d]^{i'_*} \\
    \Pi (W_{L},P) \ar[rr]^{f}_{\cong}& &\Pi (W_{L'},P)\\
    &\Pi (M,P) \ar@{->>}[ul]^{i_*} \ar@{->>}[ur]_{i_*}&
  }
\end{gather}
commutes.

Suppose moreover that the square in Diagram \eqref{d1} commutes.

Let  $j,j',u,u'$ be the inclusion maps in the diagram
\begin{gather}
  \label{e14}
  \xymatrix{
    \partial W_L\ar[r]^{u'}\ar[d]_{u}&
    W_{L'}\ar[d]_{j'}\\
    W_L\ar[r]^{j}&
    W.
    }
\end{gather}
Consider the $\pi _1$ of the above diagram
\begin{gather}
  \label{e13}
  \xymatrix{
    \pi _1\partial W_L\ar@{->>}[r]^{u'_*}\ar@{->>}[d]_{u_*}&
    \pi _1W_{L'}\ar[d]_{j'_*}^{\cong}\\
    \pi _1W_L\ar[r]^{j_*}_{\cong}\ar[ur]^{f_1}_{\cong}&
    \pi _1W.
    }
\end{gather}
Here, the square is a pushout by the Van Kampen theorem since $\partial W_L$
is connected.
The isomorphism $f_1$ is defined by
\begin{gather*}
  f_1=f\co \Pi (W_L,P)(p_1,p_1)\to\Pi (W_{L'},P)(p_1,p_1).
\end{gather*}
In \cite[Lemma 2.1]{HW}, we proved that Diagram \eqref{e13} commutes.
It follows that $j_*$ and $j'_*$ are isomorphisms.
Thus we have
\begin{gather*}
  \pi _1W\cong \pi _1W_L\cong\pi _1W_{L'}\cong \pi _1M/N_L.
\end{gather*}

Let $K(\pi _1W,1)$ be the Eilenberg--Mac Lane space, which is obtained
from $W$ by adding cells of dimension $\ge 3$.  Let
\begin{gather*}
  \rho _W\co W\rightarrow  K(\pi _1W,1)
\end{gather*}
be the inclusion map.  We set
\begin{gather*}
  \eta (M,L,L',h)=(\rho _W)_\ast ([W]) \in H_4(\pi _1W),
\end{gather*}
where $[W]\in H_4W$ is the fundamental class.  Here, and in what
follows, for a group $G$ we identify $H_*(K(G,1))$ with $H_*(G)$.

Now we state our generalization of Fenn and Rourke's theorem.  (When
$\partial M$ is connected, this is equivalent to the corresponding
case of the statement given in \cite[Theorem 4]{GK}.)

\begin{theorem}[{\cite[Theorem 2.2]{HW}}]
  \label{r7}
  Let $L$ and $L'$ be framed links in a compact, connected, oriented
  $3$--manifold $M$ with non-empty boundary.
  Then the following conditions are equivalent.
  \begin{enumerate}
  \item $L$ and $L'$ are $\delta $--equivalent.
  \item There is a diffeomorphism $h\co  M_L\ct M_{L'}$ restricting to the
    identification map $\partial M_L\cong\partial M_{L'}$, and there is a
    groupoid isomorphism
    \begin{gather*}
      f\co \Pi (W_L,P)\ct \Pi (W_{L'},P)
    \end{gather*}
    such that Diagram \eqref{d1} commutes and we have
    $\eta (M,L,L',h)=0\in H_4(\pi _1W)$.
  \end{enumerate}
\end{theorem}

Note that the statement of Theorem \ref{r7} is slightly different from
\cite[Theorem 2.2]{HW} in the following points:
\begin{itemize}
\item We use the fundamental groupoid $\Pi (M,P)$ etc. instead of
  $\pi _1(M;p_1,p_k)$ for $k=1,\ldots ,t$.
\item We use $\pi _1W$ instead of $\pi _1W_L(\cong\pi _1W)$.
\end{itemize}
These differences are not essential, and one can easily check that
Theorem \ref{r7} is equivalent to \cite[Theorem 2.2]{HW}.

\begin{remark}
  \label{r11}
  To the authors' knowledge, in the literature there have been no
  examples of the data in \cite[Theorem 6]{FR} and Theorem \ref{r7}
  with non-zero homology class in $H_4(\pi_1W)$.  In Section
  \ref{sec5}, we construct such examples for any $\alpha\in
  H_4(\pi_1W)$, see Proposition \ref{r18}.
\end{remark}

\section{$N$--links}
\label{sec3}

In this section, we give a modification of Theorem
\ref{r7} which will be useful in many situations.

We fix a normal subgroup $N$ of $\pi _1M$.
Let
\begin{gather*}
  q\co \pi _1M\twoheadrightarrow \pi _1M/N
\end{gather*}
denote the projection, which naturally extends to a full functor
\begin{gather*}
  q\co \Pi (M,P)\twoheadrightarrow \Pi (M,P)/N.
\end{gather*}

\subsection{$N$--links and surgery}
\label{sec:n-links-surgery}

A framed link $L$ in $M$ is called an {\em $N$--link} in $M$ if
$N_L\subset N$, i.e., if the homotopy class of each component of $L$ is in
$N$.

For an $N$--link $L$ in $M$, consider the following diagram
\begin{gather}
\label{e2}
  \xymatrix{
    \pi _1M_L \ar@{->>}[r]^{i'_*}\ar@{->>}[rrd]_{q_L}
    & \pi _1W_L\ar@{->>}[dr]^{\bq_L} &
    \pi _1M\ar@{->>}[l]_{i_*} \ar@{->>}[d]_{q}\\
    &&\pi _1M/N .
  }
\end{gather}
Since $N_L\subset N$,
there is a unique surjective homomorphism $\bq_L$ such that
$q=\bq_Li_*$.   We set $q_L:=\bq_L i'_*$.
Diagram \eqref{e2} naturally extends to a commutative diagram in groupoids
\begin{gather}
  \label{e8}
  \xymatrix{
    \Pi (M_L,P_L) \ar@{->>}[r]^{i'_*}\ar@{->>}[rrd]_{q_L}
    & \Pi (W_L,P)\ar@{->>}[dr]^{\bq_L} &
    \Pi (M,P)\ar@{->>}[l]_{i_*} \ar@{->>}[d]_{q}\\
    &&\Pi (M,P)/N .
  }
\end{gather}

Suppose that $L$ and $L'$ are $N$--links in $M$ and $h\co M_L\ct M_{L'}$
is a diffeomorphism restricting to the identification map
$\partial M_L\cong\partial M_{L'}$ such that the following diagram commutes.
\begin{equation}
  \label{e4}
  \xymatrix{
    \Pi (M_L,P_L)\ar@{>>}[dr]_{q_L} \ar@{>}[rr]_{\cong}^{h_*}
    && \Pi (M_{L'},P_{L'})\ar@{>>}[dl]^{q_{L'}}\\
    &\Pi (M,P)/N.
  }
\end{equation}

\begin{lemma}
  \label{r16}
  In the above situation,
  there is a unique functor
  \begin{gather*}
    g\co \Pi (W,P)\twoheadrightarrow\Pi (M,P)/N
  \end{gather*}
  such that $\bq_L=gj_*$ and $\bq_{L'}=gj'_*$.
\end{lemma}

\begin{proof}
  Consider the following diagram.
  \begin{gather}
    \label{e16}
    \xymatrix{
      \Pi (M_L,P_L)
      \ar@/^1pc/[rr]^{h_*}
      \ar[dr]^{v_*}
      \ar[dd]_{(i')_*^{L}}
      \ar `l[d] +/l3.5pc/`[ddddr][ddddr]^{q_L}
      &
      \Pi (M,P)
      \ar[d]_{k_*}
      \ar[ddl]_{i_*^L}
      \ar[ddr]^{i_*^{L'}}
      &
      \Pi (M_{L'},P_{L'})
      \ar[dl]_{v'_*}
      \ar[dd]^{(i')_*^{L'}}
      \ar `r[d] +/r3.5pc/`[ddddl][ddddl]_{q_{L'}}
      \\
      &
      \Pi (\partial W_L,P)
      \ar[dl]^{u_*}
      \ar[dr]_{u_*'}
      &
      \\
      \Pi (W_L,P)
      \ar[dr]^{j_*}
      \ar[ddr]_{\bq_L}
      &&
      \Pi (W_{L'},P)
      \ar[dl]_{j'_*}
      \ar[ddl]^{\bq_{L'}}
      \\
      &
      \Pi (W,P)
      \ar@{.>}[d]_{g}
      &
      \\
      &
      \Pi (M,P)/N
      &
    }
  \end{gather}
  The arrows above $\Pi (W,P)$ are induced by a commutative diagram of
  inclusions of submanifolds of $W$, and hence form a commutative subdiagram.
  The middle diamond $j_*u_*=j'_*u'_*$ is a pushout.
  Therefore, to prove existence of $g\co \Pi (W,P)\rightarrow \Pi (M,P)/N$ which makes the above
  diagram commute (i.e., $gj_*=\bq_L$ and $gj'_*=\bq_{L'}$), it
  suffices to prove that $\bq_Lu_*=\bq_{L'}u'_*$.  Since the groupoid
  $\Pi (\partial W_L,P)$ is generated by the images of $k_*$ and $v_*$, it
  suffices to check that
  \begin{gather*}
    \bq_Lu_*k_*=\bq_{L'}u'_*k_*,\quad
    \bq_Lu_*v_*=\bq_{L'}u'_*v_*
  \end{gather*}
  Indeed, we have
  \begin{gather*}
    \bq_Lu_*k_*
    =\bq_Li_*^L
    =\bq_{L'}i_*^{L'}
    =\bq_{L'}u'_*k_*,
  \end{gather*}
  and
  \begin{gather*}
    \bq_Lu_*v_*
    =\bq_L(i')^{L}_*=q_L=q_{L'}h_*
    =\bq_{L'}(i')^{L'}_*h_*
    =\bq_{L'}u'_*v'_*h_*
    =\bq_{L'}u'_*v_*.
  \end{gather*}
\end{proof}

By Lemma \ref{r16}, there is a surjective homomorphism
\begin{gather*}
  g\co \pi _1W\twoheadrightarrow\pi _1M/N.
\end{gather*}
Let $K(\pi _1M/N,1)$ be obtained from $K(\pi _1W,1)$ by attaching cells of
dimension $\ge 2$.
Let
\begin{gather*}
  \rho _{W,N}\co W \rightarrow K(\pi _1M/N,1)
\end{gather*}
be the inclusion map.  Now, define a homology class
\begin{gather*}
  \eta _{\pi _1M/N}(M,L,L',h)=(\rho _{W,N})_*([W])\in H_4(\pi _1M/N).
\end{gather*}

\subsection{$K_3(N)$--moves}
\label{sec:n-moves}

A {\em $K_3(N)$--move} on a framed link in $M$ is a $K_3$--move
$L\leftrightarrow L\cup K\cup K'$ as in Figure \ref{K3move}, where
the homotopy class of $K$ is contained in $N$.  Note that a
$K_3(N)$--move on an $N$--link produces another $N$--link.

The {\em $\delta (N)$--equivalence} on $N$--links in $M$ is defined as the
equivalence relation generated by stabilizations, handle-slides and
$K_3(N)$--moves.

Suppose that $L$ and $L\cup K\cup K'$ are related by a $K_3(N)$--move as
above.  Let $V$ be a tubular neighborhood of $K$ in $M\setminus L$ containing
$K'$ in the interior.  Then there is a diffeomorphism $h\co M_L\cong
M_{L\cup K\cup K'}$ restricting to the identity on $M\setminus \opint V$.  Such an
$h$ is unique up to isotopy relative to $M\setminus\opint V$.  The
$4$--manifold $W_{L\cup K\cup K'}$ is diffeomorphic to the $4$--manifold obtained
from $W_L$ by surgery along the framed knot
\begin{gather*}
  \tilde{K}:=K\times \{1/2\}\subset M\times [0,1]\subset W_L\left(=M\times [0,1]\cup (\text{$2$--handles})\right),
\end{gather*}
where the framing of $\tilde{K}$ is determined by that of $K$.
(This fact follows since $(M\times[0,1])_{\tilde{K}}$ is diffeomorphic
to
the relative double $D(W_K)=W_K\cup_{M_K} (-W_K)$ of the $4$--manifold $W_K$, and
the $2$--handle in $W_K$ induces the dual $2$--handle in $D(W_K)$,
which is attached along a $0$--framed meridian to $K$, i.e., $K'$.)
Thus, there is a natural surjective homomorphism
\begin{equation*}
  \theta \co \pi _1W_L\twoheadrightarrow\pi _1W_{L\cup K\cup K'}
\end{equation*}
with kernel normally generated by the homotopy class of $\tilde{K}$.
We have a cobordism
\begin{gather*}
  X:=(W_L\times [0,1])\cup (\text{a $2$--handle attached along $\tilde{K}\times \{1\}$})
\end{gather*}
between $W_L$ and $(W_L)_{\tilde{K}}\cong W_{L\cup K\cup K'}$.  This
cobordism $X$ is over $K(\pi _1M/N,1)$ since $\tilde{K}$ maps to a
null-homotopic loop in $K(\pi _1M/N,1)$.

The homomorphism $\theta$ extends in a natural way to a full,
identity-on-objects functor
\begin{equation*}
  \theta \co \Pi (W_L,P)\twoheadrightarrow\Pi (W_{L\cup K\cup K'},P).
\end{equation*}

\begin{lemma}
  \label{r15}
  In the above situation with $L':=L\cup K\cup K'$, Diagram \eqref{e4}
  commutes and we have
  \begin{gather}
    \label{e17}
    \eta _{\pi _1M/N}(M,L,L',h)=0\in H_4(\pi _1M/N).
  \end{gather}
\end{lemma}

\begin{proof}
To prove commutativity of \eqref{e4}, consider the following diagram.
\begin{equation}
  \label{e12}
  \xymatrix{
    \Pi (M_L,P_L) \ar@{>>}[d]^{(i')^L_\ast}\ar[rr]^{h_\ast}_{\cong}
	\ar@{>>} `l[d] +/l3.5pc/`[dddr][dddr]^{q_L}
    &&
    \Pi (M_{L'},P_{L'}) \ar@{>>}[d]_{(i')^{L'}_\ast}
     \ar@{>>} `r[d] +/r3.5pc/`[ddd] [dddl]_{q_{L'}}
    &\\
    \Pi (W_L,P) \ar@{>>}[rr]^{\theta }
    \ar@{>>}[ddr]_{\bq_L}
    &&
    \Pi (W_{L'},P)
    \ar@{>>}[ddl]^{\bq_{L'}}
    &\\
    &\Pi (M,P) \ar@{>>}[ul]_{i^L_*} \ar@{>>}[ur]^{i^{L'}_\ast}
    \ar@{>>}[d]_{q}
    &\\
    &
    \Pi (M,P)/N
    &
  }
  \end{equation}
We have
\begin{equation*}
  q_{L'}h_*
  =\bq_{L'}(i')^{L'}_*h_*
  =\bq_{L'}\theta (i')^L_*
  =\bq_L(i')^L_*
  =q_L.
\end{equation*}
Here we have $\bq_{L'}\theta =\bq_L$ since
we have
\begin{equation*}
  \bq_{L'}\theta i^L_*
  =\bq_{L'}i^{L'}_*
  =q
  =\bq_L i^L_*
\end{equation*}
and the functor $i^L_*$ is full and the identity on objects.

Now, we will prove \eqref{e17}.
As we have observed, $W_L$ and $W_{L'}=W_{L\cup K\cup K'}$ are cobordant over
$K(\pi _1M/N,1)$.
Hence we have
\begin{gather*}
  \eta _{\pi _1M/N}(M,L,L',h)=(\rho _{W,N})_*([W])=0.
\end{gather*}
\end{proof}

\subsection{Characterization of $\delta (N)$--equivalence}
\label{sec:char-n-equiv}

We have the following characterization of the $\delta (N)$--equivalence.

\begin{theorem}
  \label{t2}
  Let $M$ be a compact, connected, oriented $3$--manifold with
  non-empty boundary.  Let $P\subset \partial M$ contain exactly one
  point of each connected component of $\partial M$.  Let $N$ be a
  normal subgroup of $\pi _1M$.  Let $L$ and $L'$ be $N$--links in
  $M$.  Then the following conditions are equivalent.
  \begin{enumerate}
  \item $L$ and $L'$ are $\delta (N)$--equivalent.
  \item There is a diffeomorphism $h\co M_L\cong M_{L'}$
    restricting to the identification map $\partial M_L\cong\partial M_{L'}$ such that
    Diagram \eqref{e4} commutes and we have
    \begin{gather}
      \label{e9}
      \eta _{\pi _1M/N}(M,L,L',h)=0\in H_4(\pi _1M/N).
    \end{gather}
  \end{enumerate}
\end{theorem}

\begin{proof}[Proof Theorem \ref{t2}, ``only if'' part]
This part follows from Lemma \ref{r15} and the ``only if''
part of Theorem \ref{r7}.
\end{proof}

For the ``if'' part, we first consider the case where $N$ is normally
finitely generated in $\pi _1M$.

\begin{proof}[Proof of Theorem \ref{t2}, ``if'' part with $N$ normally finitely generated in $\pi _1M$]
  By the assumption, there is a framed link $K=K_1\cup \dots \cup K_k$ in $M$
  disjoint from both $L$ and $L'$ such that $N_K=N$.
  Let $K^*=K_1^*\cup \dots \cup K_k^*$ be a framed link in $M$ consisting of
  small $0$--framed meridians $K_j^*$ to
  $K_j$.  Thus $L$ and $\tL:=L\cup K\cup K^*$ (resp. $L'$ and
  $\tL':=L'\cup K\cup K^*$) are related by $k$ $K_3(N)$--moves.  We have
  $N=N_{\tL}=N_{\tL'}$.

  It suffices to prove that $\tilde L$ and $\tilde L'$ are
  $\delta $--equivalent.  Consider the following diagram.

  \begin{equation}
    \label{e26}
  \xymatrix@C=3.8mm{
    \Pi (M_L,P_L) \ar[r]^{m_*}_\cong
    \ar@{>>}[d]^{(i')^L_*} \ar@<1.5ex>@/^1.5pc/[rrrr]^{h_*}_\cong
   	\ar@{>>} `l[d] +`[dddrr][dddrr]^{q_L}
    &
    \Pi (M_{\tL},P_{\tL})
      \ar[rr]^{\tilde{h}_*=(m'hm^{-1})_*}_\cong
      \ar@{>>}[d]^{(i')^{\tL}_*}
      &
      &
      \Pi (M_{\tL'},P_{\tL'})
      \ar@{>>}[d]_{(i')^{\tL'}_*}
      &
      \Pi (M_{L'},P_{L'})
      \ar[l]_{m'_*}^\cong
      \ar@{>>}[d]_{(i')^{L'}_*}
     \ar@{>>} `r[d] +/r3pc/`[ddd] [dddll]_{q_{L'}}
      \\
      \Pi (W_L,P)
      \ar@{>>}[r]^{\theta }
      \ar@{>>}[rrdd]_{\bq_L}
      &
      \Pi (W_{\tL},P)
      \ar[rr]^{f:=\bq_{\tL'}^{-1}\bq_{\tL}}_\cong
      \ar[rdd]_{\bq_{\tL}}^\cong
      & &
      \Pi (W_{\tL'},P)
      \ar[ldd]^{\bq_{\tL'}}_\cong
      &
      \Pi (W_{L'},P)
      \ar@{>>}[l]_{\theta '}
      \ar@{>>}[lldd]^{\bq_{L'}}
      \\
      &
      &
      \Pi (M,P)
      \ar@{>>}[ul]_{i_*^{\tL}}
      \ar@{>>}[ur]^{i_*^{\tL'}}
      \ar@{>>}[d]_q
      &
      &
      \\
      &
      &
      \Pi (M,P)/N
      &
      &
      }
  \end{equation}

  Here $m\co M_L\ct M_{\tL}$ and $m'\co M_{L'}\ct M_{\tL'}$ are
  natural diffeomorphisms, and we set
  $\tilde{h}=m'hm^{-1}\co M_{\tL}\ct M_{\tL'}$.  All the faces
  except the middle square commute.  Since the outermost triangle
  commutes, i.e., $q_L=q_{L'}h_*$, one can check that
  \begin{gather*}
    \bq_{\tL'}f(i')^{\tL}_*=\bq_{\tL'}(i')^{\tL'}_*\tilde{h}_*.
  \end{gather*}
  Since $\bq_{\tL'}$ is an isomorphism, the middle square commutes, i.e.,
  \begin{gather*}
    f(i')^{\tL}_*=(i')^{\tL'}_*\tilde{h}_*.
  \end{gather*}
  Thus, the whole Diagram \eqref{e26} commutes.

  Set
  \begin{gather*}
    W:=W_{M,L,L',h}= W_{L}\cup _\partial (-W_{L'}),\\
    \tilde{W}:=W_{M,\tL,\tL',\tilde{h}}= W_{\tL}\cup _\partial (-W_{\tL'}).
  \end{gather*}
  By commutativity of the middle pentagon, the homology class
  \begin{gather*}
    \eta _{\pi _1M/N}(M,\tL,\tL',\tilde{h})=(\rho _{\tilde{W}})_*([\tilde{W}])\in H_4(\pi _1M/N),
  \end{gather*}
  is defined.  We claim that $\tilde{W}$ and $W$ are bordant over
  $K(\pi _1M/N)$.  Indeed, there is an oriented, compact
  $5$--cobordism $X$ between $W$ and $\tilde{W}$ constructed as in
  Section \ref{sec:n-moves}, which maps to $K(\pi _1M/N,1)$.  Hence it
  follows that
  \begin{gather*}
    \eta_{\pi_1M/N}(M,\tilde{L},\tilde{L}',\tilde{h})
    =(\rho_{\tilde{W}})_*([\tilde{W}])
    =(\rho_{{W}})_*([\tilde{W}])
    =\eta_{\pi_1M/N}(M,{L},{L}',{h}),
  \end{gather*}
  which is $0$ by the assumption.
  Then, by Theorem \ref{r7}, it follows that $\tL$ and $\tL'$ are $\delta $--equivalent.
\end{proof}

\begin{proof}[Proof of Theorem \ref{t2}, ``if'' part (general case)]
  Let $N_0\subset N$ denote the smallest normal subgroup in $\pi _1M$
  containing $N_L\cup N_{L'}$.
  Let $$q^0\co \Pi (M,P)\onto\Pi (M,P)/N_0$$ be the projection.
  Let $$\bq^0_L\co \Pi (W_L,P)\onto\Pi (M,P)/N_0$$ be the functor such
  that $q^0=\bq^0_L i^L_*$.  Set
  $$q^0_L=\bq^0_L(i')^L_*\co \Pi (M_L,P_L)\onto\Pi (M,P)/N_0.$$
  Similarly, define $$\bq^0_{L'}\co \Pi (W_{L'},P)\onto\Pi (M,P)/N_0$$ and
  $$q^0_{L}\co \Pi (M_{L'},P_{L'})\onto\Pi (M,P)/N_0.$$

  Let $\bar{N}_1\subset \pi _1M/N_0$ be the normal subgroup generated by
  the elements
  \begin{gather*}
    q^0_L(a)^{-1}\cdot q^0_{L'}(h_*(a))
  \end{gather*}
  for $a\in \pi _1M_L$.  By $q_L=q_{L'}h_*$, it follows that $\bar{N}_1\subset N/N_0$.
  Since $\pi _1M_L$ is finitely generated, it follows
  that $\bar{N}_1$ is normally finitely generated in $\pi _1M/N_0$.
  Set
  \begin{gather*}
    N_1=(q^0)^{-1}(\bar{N}_1)\subset N,
  \end{gather*}
  which is normally finitely generated in $\pi _1M$.

  Let $p_{N_0,N_1}\co \Pi (M,P)/N_0\onto\Pi (M,P)/N_1$ be the projection.
  Set
  \begin{gather*}
    q^1_L=p_{N_0,N_1}q^0_L\co \Pi (M_L,P_L)\onto\Pi (M,P)/N_1,\\
    q^1_{L'}=p_{N_0,N_1}q^0_{L'}\co \Pi (M_{L'},P_{L'})\onto\Pi (M,P)/N_1.
  \end{gather*}

  We have $q^1_L=q^1_{L'}h_*$.  Hence we have a well-defined homology
  class
  \begin{gather*}
    \eta _{\pi _1M/N_1}(M,L,L',h)\in H_4(\pi _1M/N_1).
  \end{gather*}
  Since $N$ is a union of normally finitely generated subgroups of
  $\pi _1M$ and homology preserves direct limits, it follows that there
  is a normally finitely generated subgroup $N_2$ of $\pi _1M$ such that
  $N_1\subset N_2\subset N$ and
  \begin{gather*}
    (p_{N_1,N_2})_*(\eta _{\pi _1M/N_1}(M,L,L',h))=\eta _{\pi _1M/N_2}(M,L,L',h)=0\in H_4(\pi _1M/N_2),
  \end{gather*}
  where $p_{N_1,N_2}\co \Pi (M,P)/N_1\onto\Pi (M,P)/N_2$ is the projection.
  The following triangle commutes
  \begin{gather*}
    \xymatrix{
      \Pi (M_L,P_L)
      \ar[rr]_{\cong}^{h_*}
      \ar@{>>}[rd]_{q^2_{L}}
      &&
      \Pi (M_{L'},P_{L'})
      \ar@{>>}[ld]^{q^2_{L'}}
      \\
      &
      \Pi (M,P)/N_2,
      }
  \end{gather*}
  where $q^2_L=p_{N_1,N_2}q^1_L$ and $q^2_{L'}=p_{N_1,N_2}q^1_{L'}$.
  Now we can apply the above-proved case of the theorem to deduce that $L$ and $L'$
  are $\delta (N_2)$--equivalent.  Hence they are $\delta (N)$--equivalent.
\end{proof}

\section{Manifolds over $K(G,1)$}
\label{sec4}

\subsection{Bordism groups}
\label{sec:bordism-groups}

Fix a group $G$.  Let $K(G,1)$ denote the Eilenberg--Mac Lane space.

By an {\em $n$--manifold over $K(G,1)$} or {\em $G$--$n$--manifold} we
mean a pair $(M,\rho _M)$ of a compact, oriented, smooth $n$--manifold $M$
and a map $\rho _M\co M\rightarrow K(G,1)$.  Here we require no condition about the
basepoints even when $M$ has a specified basepoint.  A
$G$--$n$--manifold $(M,\rho _M)$ will often be simply denoted by $M$.

For $n\ge 0$, let $\Omega _n(G)=\Omega _n(K(G,1))$ denote the $n$--dimensional
oriented bordism group of $K(G,1)$, which is defined to be the set of
bordism classes of closed $G$--$n$--manifolds.

There is a natural map
\begin{gather*}
  \theta _n\co \Omega _n(G)\rightarrow H_n(G)
\end{gather*}
defined by
\begin{gather*}
  \theta _n([M,\rho _M])=(\rho _M)_*([M])\in H_n(G).
\end{gather*}
As is well known (see e.g.\ \cite[Section 2]{HK}), one can use
$\Omega_0=\Omega_4=\Z$, $\Omega_1=\Omega_2=\Omega_3=0$ and the
Atiyah--Hirzebruch spectral sequence to show that $\theta _n$ is an
isomorphism for $n=0,1,2,3$, and we have an isomorphism
\begin{gather}
  \label{e5}
  \left(\begin{matrix}
    \theta _4\\\sigma
  \end{matrix}\right)\co \Omega _4(G)\ct H_4(G) \oplus \Z,
\end{gather}
where
\begin{gather*}
  \sigma ([M,\rho _M])=\text{signature}(M)\in \Z  .
\end{gather*}

\subsection{$G$--surfaces, bordered $3$--manifolds and cobordisms}
\label{sec:g-surfaces-bordered}
By a {\em $G$--surface} we mean a closed $G$--$2$--manifold.

Let $(\Sigma ,\rho _\Sigma )$ be a $G$--surface.  A {\em
$(\Sigma ,\rho _\Sigma )$--bordered $3$--manifold} will mean a triple $(M,\rho _M,\phi _M)$
such that $(M,\rho _M)$ is a $G$--$3$--manifold and $\phi _M\co \Sigma \ct\partial M$ is an
orientation-preserving diffeomorphism satisfying
$\rho _\Sigma =(\rho _M|_{\partial M})\phi _M$.

A {\em cobordism} between $(\Sigma ,\rho _\Sigma )$--bordered $3$--manifolds
$\triM$ and $\tri{M'}$ is a triple $\tri{W}$ consisting of a
$G$--$4$--manifold $(W,\rho _W)$ and an orientation-preserving diffeomorphism
\begin{gather*}
  \phi _W\co M\cup _\Sigma (-M')\ct \partial W,
\end{gather*}
where $M\cup _\Sigma (-M')$ is the closed oriented $4$--manifold obtained by
gluing $M$ and $-M'$ along their boundaries using the diffeomorphism
$\phi _{M'}\phi _M^{-1}\co \partial M\ct \partial M'$, such that
the following diagram commutes
\begin{gather*}
  \xymatrix{
    M
    \ar[d]_{\text{incl}}
    \ar[drr]^{\rho _M}
    &&\\
    M\cup _\Sigma (-M')
    \ar[r]^-{\phi _W}
    &
    W
    \ar[r]^{\rho _W\quad \quad }
    &
    K(G,1).
    \\
    M'
    \ar[u]^{\text{incl}}
    \ar[urr]_{\rho _{M'}}
    &&
    }
\end{gather*}
We denote this situation by $\triW\co \tri{M}\rightarrow \tri{M'}$ or simply by
$W\co M\rightarrow M'$.

Two cobordisms $W,W'\co M\rightarrow M'$ between $(\Sigma ,\rho _\Sigma )$--bordered
$3$--manifolds $M=(M,\rho _M)$ and $M'=(M',\rho _{M'})$ are said to be {\em
cobordant} if there is a {\em cobordism} between them, i.e., a triple
$\tri{X}$ consisting of a $G$--$5$--manifold $X=(X,\rho _X)$ and an
orientation-preserving diffeomorphism
\begin{gather*}
  \phi _X\co W'' \ct \partial X,
\end{gather*}
where $W'':=W\cup _{M\cup _\Sigma (-M')}(-W')$ is the closed, oriented
$4$--manifold obtained from $W$ and $-W'$ by gluing along
${M\cup _\Sigma (-M')}$ using the diffeomorphism $\phi _{W'}\phi _W^{-1}\co W\rightarrow W'$,
such that the following diagram commutes
\begin{gather*}
  \xymatrix{
    W
    \ar[d]_{\text{incl}}
    \ar[drr]^{\rho _W}
    &&\\
    W''
    \ar[r]^-{\phi _X}
    &
    X
    \ar[r]^{\rho _X\quad \quad }
    &
    K(G,1).
    \\
    W'
    \ar[u]^{\text{incl}}
    \ar[urr]_{\rho _{W'}}
    &&
    }
\end{gather*}

\subsection{Cobordism groupoid $\modC =\modC _{(\Sigma ,\rho _\Sigma )}$}
\label{sec:cobordism-groupoid-=}
As in the last subsection, let $\Sigma =(\Sigma ,\rho _\Sigma )$ be a $G$--surface.

For our purpose, it is convenient to introduce the category
$\modC =\modC _{(\Sigma ,\rho _\Sigma )}$ of $\Sigma $--bordered $3$--manifolds and cobordism
classes of cobordisms between $\Sigma $--bordered $3$--manifolds, defined as
follows.

The objects in $\modC $ are $\Sigma $--bordered $3$--manifolds.  The morphisms
between two $\Sigma $--bordered $3$--manifolds $M=\triM$ and
$M'=\tri{M'}$ are the cobordism classes of cobordisms between
$M$ and $M'$.

The composition in $\modC $ is induced by the composition of cobordisms
defined below.  Two cobordisms $W\co M\rightarrow M'$ and $W'\co M'\rightarrow M''$ can be
composed in the usual way: let $W'\circ W=W'\cup _{M'}W$ be the
$4$--manifold obtained by gluing $W'$ and $W$ along $M'$ using the
maps
\begin{gather*}
  \phi _{W}|_{M'}\co M'\to W,\quad
  \phi _{W'}|_{M'}\co M'\to W'.
\end{gather*}
Let
\begin{gather*}
  \rho _{W'\circ W} = \rho _{W'}\cup \rho _W\co W'\circ W\rightarrow K(G,1),
\end{gather*}
and
\begin{gather*}
  \phi _{W'\circ W}= (\phi _{W'}|_{M''})\cup (\phi _{W}|_{M})\co M\cup _{\Sigma }(-M'')\ct
  \partial (W'\circ W).
\end{gather*}
Then we obtain a new cobordism over $K(G,1)$
\begin{gather*}
  {W'\circ W}=\tri{W'\circ W}\co \tri{M}\rightarrow \tri{M''}.
\end{gather*}

The identity morphism $1_M\co M\rightarrow M$ is represented by the ``reduced''
cylinder $C_M=\tri{C_M}$.
The $4$--manifold $C_M$ is defined by
\begin{gather}
  \label{e-reduced-cylinder}
  C_M = (M\times [0,1])/\sim,
\end{gather}
where $\sim$ is generated by $(x,t)\sim(x,t')$ for $x\in \partial M$ and
$t,t'\in [0,1]$.
The map $\rho _{C_M}\co C_M\rightarrow K(G,1)$ is induced by the composite
\begin{gather*}
  M\times [0,1]\overset{\text{proj}}{\longrightarrow} M \xto{\rho _M}K(G,1).
\end{gather*}
The map $\phi _{C_M}\co M\cup _\Sigma (-M)\rightarrow \partial C_M$ is given by
\begin{gather*}
  \phi _{C_M} = \phi _{M,\partial C_M}\cup \phi _{-M,\partial C_M},
\end{gather*}
where $\phi _{M,\partial C_M}\co M\hookrightarrow \partial C_M$ is induced by $M\cong M\times \{1\}\hookrightarrow
M\times [0,1]$, and $\phi _{-M,\partial C_M}\co (-M)\hookrightarrow \partial C_M$ is induced by $M\cong
M\times \{0\}\hookrightarrow M\times [0,1]$.

It is not difficult to check that the above definition gives a
well-defined category.

By abuse of notation, the morphism in $\modC $ represented by a cobordism
$W=\triW$ from $M$ to $M'$ is again denoted by ${W}=\tri{W}$.

The category $\modC $ is a groupoid by the same reason that $\Omega _n(G)$ is a
group.  Indeed, for a morphism ${W}=\tri{W}\co \triM\rightarrow \tri{M'}$ in
$\modC $, the inverse ${W}^{-1}$ is represented by the cobordism
\begin{gather*}
  W^{-1}:=(-W,\rho _{W^{-1}},\phi _{W^{-1}})\co \tri{M'}\rightarrow \tri{M},
\end{gather*}
where $\rho _{W^{-1}}=\rho _W\co (-W)\rightarrow K(G,1)$, and
$\phi _{W^{-1}}\co M'\cup _{\Sigma }(-M)\ct\partial (-W)$ is the composite
\begin{gather*}
  M'\cup _{\Sigma }(-M) \cong
  -(M\cup _{\Sigma }(-M))\overset{-\phi _W}{\underset{\cong}{\longrightarrow}}(-\partial W)\cong \partial (-W).
\end{gather*}
The composite $W^{-1}\circ W$ is cobordant to $C_M$ via a cobordism
$\tri{X}$.  Here the $5$--manifold $X$ is the ``partially reduced
cylinder''
\begin{gather*}
  X:=(W\times [0,1])/\sim,
\end{gather*}
where $(\phi _W(x),t)\sim(\phi _W(x),t')$ for $x\in (-M')\subset M\cup
_\partial (-M')$, $t\in [0,1]$.  The map $\rho _X\co X\rightarrow
K(G,1)$ is induced by the composite
\begin{gather*}
  W\times [0,1]\xto{\text{proj}}W \xto{\rho _W}K(G,1).
\end{gather*}
The diffeomorphism $$\phi _X\co (W^{-1}\circ W)\cup _\partial
(-C_M)\rightarrow \partial X$$ is given by
\begin{gather*}
  \phi _X(w)=(w,0)\quad \text{for $w\in W\subset W^{-1}\circ W$},\\
  \phi _X(w)=(w,1)\quad \text{for $w\in -W\subset W^{-1}\circ W$},\\
  \phi _X([x,t])=(i_{M,W},t)\quad \text{for $x\in M$, $t\in [0,1]$}.
\end{gather*}
Here $[x,t]\in C_M$ is represented by $(x,t)\in M\times [0,1]$, and
$i_{M,W}\co M\to W$ is the composite
\begin{gather*}
  M \subset  M\cup _\partial (-M')\xto{\phi _{W}}\partial W\subset W.
\end{gather*}
Similarly, $W\circ W^{-1}$ is cobordant to $C_{M'}$.  Thus $W\co M\rightarrow M'$
is an isomorphism in~$\modC $.

\subsection{$G$--diffeomorphism}
\label{sec:g-diffeomorphism}
Let $\triM$ and $\tri{M'}$ be two
$(\Sigma ,\rho _\Sigma )$--bordered $3$--manifolds.

By a {\em $G$--diffeomorphism} $$h\co \triM\ct \tri{M'}$$ we mean a
diffeomorphism $h\co M\ct M'$ such that
\begin{enumerate}
\item $h$ is compatible with the maps $\phi _M\co \Sigma \ct \partial M$ and
  $\phi _{M'}\co \Sigma \ct\partial M'$, i.e., we have $\phi _{M'}=(h|_{\partial M})\phi _M$,
\item $h$ is compatible with the maps $\rho _M\co M\rightarrow K(G,1)$ and
  $\rho _{M'}\co M'\rightarrow K(G,1)$ {\em up to homotopy relative to
  $\partial M$}, i.e., we have
  \begin{gather}
    \label{e19}
    \rho _M{\simeq}\rho _{M'}h\co M\rightarrow K(G,1)\quad (\text{rel $\partial M$}).
  \end{gather}
\end{enumerate}
In this case, $\triM$ and $\tri{M'}$ are said to be {\em
  $G$--diffeomorphic}.

We have the following characterization of $G$--diffeomorphism in terms
of fundamental groupoids.

\begin{proposition}
  \label{r21}
  Let $(\Sigma ,\rho _\Sigma )$ be a non-empty $G$--surface.
  Let $\triM$ and $\tri{M'}$ be connected $(\Sigma ,\rho _\Sigma )$--bordered
  $3$--manifolds.  Let $P_\Sigma \subset \Sigma $ be a subset containing exactly one
  point of each connected component of $\Sigma $, and set
  $P=\phi _M(P_\Sigma )\subset \partial M$ and $P'=\phi _{M'}(P_\Sigma )\subset \partial M'$.  Then the
  following conditions are equivalent.
  \begin{enumerate}
  \item $\triM$ and $\tri{M'}$ are $G$--diffeomorphic.
  \item There is a diffeomorphism $h\co M\ct M'$ compatible with the
  maps $\phi _M$ and $\phi _{M'}$ such that the following groupoid diagram
  commutes
    \begin{gather}
      \label{e22}
      \xymatrix{
	\Pi (M,P)
	\ar[rr]^{h_*}_{\cong}
	\ar[dr]_{(\rho _M)_*}
	&&
	\Pi (M',P')
	\ar[dl]^{(\rho _{M'})_*}
	\\
	&
	\Pi (K(G,1),\rho _\Sigma (P_\Sigma )).
	&
      }
    \end{gather}
  \end{enumerate}
\end{proposition}

\begin{proof}
  It is clear that (i) implies (ii).

  Suppose that (ii) holds.  It suffices to prove \eqref{e19}.  Suppose
  $p_1,\ldots ,p_t$ $(t\ge 1)$ be the elements of $P_\Sigma $.  For $i=2,\ldots ,t$,
  let $\gamma _i$ be a simple curve between $\phi _M(p_1)$ and $\phi _M(p_i)$ in
  $M$ such that $\gamma _i\cap \gamma _j=\{\phi _M(p_1)\}$ if $i\neq j$.
  Commutativity of \eqref{e22} implies that, for each $i=2,\ldots ,t$, the
  maps $\rho _{M}|_{\gamma _i}\co \gamma _i\rightarrow K(G,1)$ and
  $(\rho _{M'}h)|_{\gamma _i}\co \gamma _i\rightarrow K(G,1)$ are
  homotopic relative to endpoints.
  Hence $\rho _M$ is homotopic rel $\partial M$ to a map
  $(\rho _M)_1\co M\rightarrow K(G,1)$ such that
  \begin{gather*}
    (\rho _M)_1|_{\partial M\cup \gamma _2\cup \dots \cup \gamma _t}=(\rho _{M'}h)|_{\partial M\cup \gamma _2\cup \dots \cup \gamma _t}.
  \end{gather*}
  Note that the subcomplex ${\partial M\cup \gamma _2\cup \dots \cup \gamma _t}$ of $M$ is
  connected.
  By \eqref{e22} we have the following commutative diagram.
  \begin{gather}
    \label{e15}
    \xymatrix{
      \pi _1(M,\phi _M(p_1))
      \ar[rr]^{h_*}_{\cong}
      \ar[dr]_{(\rho _M)_*}
      &&
      \pi _1(M',\phi _{M'}(p_1))
      \ar[dl]^{(\rho _{M'})_*}
      \\
      &
      \pi _1(K(G,1),\rho _\Sigma (p_1))=G.
      &
    }
  \end{gather}
  By the property of the Eilenberg--Mac Lane space $K(G,1)$, it
  follows that $(\rho _M)_1$ is homotopic rel ${\partial M\cup \gamma
  _2\cup \dots \cup \gamma _t}$ to $\rho _{M'}h$.  (Here we use the
  following fact.  Let $X$ be a connected CW complex and let $Y\subset
  X$ be a connected subcomplex.  Suppose $f,f'\co X\rightarrow K(G,1)$
  be maps such that $f|_Y=f'|_Y$ and $f_*=f'_*\co \pi _1X\rightarrow
  \pi _1(K(G,1))=G$.  Then $f$ and $f'$ are homotopic rel $Y$.)
\end{proof}

\subsection{Mapping cylinder}
\label{sec:mapping-cylinder}
Let $h\co \triM\ct \tri{M'}$ be a $G$--diffeomorphism of
$(\Sigma ,\rho _\Sigma )$--bordered $3$--manifolds.

As before, let $C_M=\tri{C_M}$ denote the reduced cylinder over $M$,
which is a cobordism from $M$ to itself.

A {\em mapping cylinder} associated to $h$ is a cobordism
\begin{gather*}
  C_h=(C_M, \rho _{C_h}, \phi _{C_h})\co \triM \rightarrow  \tri{M'}
\end{gather*}
defined as follows.  The map
\begin{gather*}
  \rho _{C_h}\co C_M\rightarrow K(G,1)
\end{gather*}
is induced by a homotopy
\begin{gather*}
  \rho _{\tilde{C}_h}\co M\times [0,1]\rightarrow K(G,1)
\end{gather*}
realizing \eqref{e19}.  The map $\rho _{C_h}$ is well defined since
\begin{gather*}
  \rho _{\tilde{C}_h}(x,0)=\rho _M(x),\quad
  \rho _{\tilde{C}_h}(x,1)=\rho _{M'}(x),\quad
  \rho _{\tilde{C}_h}(y,t)=\rho _M(y)
\end{gather*}
for $x\in M$, $y\in \partial M$, $t\in [0,1]$.
The map
\begin{gather*}
  \phi _{C_h}\co M\cup _\Sigma (-M')\ct \partial C_h
\end{gather*}
is obtained by gluing two diffeomorphisms
\begin{gather*}
  M\ct M\times \{0\},\quad \text{and}\quad M'\xtoc{h^{-1}} M\ct M\times \{1\}.
\end{gather*}

By the property of $K(G,1)$, it follows that $C_h$ defines a unique
morphism from $M$ to $M'$ in $\modC $.

\subsection{Closure map}
\label{sec:closure-map}

Let $W=\tri{W}\co M\rightarrow M$ be an endomorphism of
$M=\triM\in\Ob(\modC )$.
Note that $\phi _W\co M\cup _\Sigma (-M)\ct\partial W$.

 Let $\hat{W}$ denote the closed $4$--manifold
obtained from $W$ by identifying $\phi _W(M)\subset \partial W$ and $\phi _W(-M)\subset \partial W$ by
the diffeomorphism $(\phi _W|_{-M})\circ(\phi _W|_M)^{-1}$.  The map
$\rho _W\co W\rightarrow K(G,1)$ induces a map
\begin{gather*}
  \rho _{\hat{W}}\co \hat{W}\rightarrow K(G,1).
\end{gather*}
Set
\begin{gather*}
  \tr(W)=[\hat{W},\rho _{\hat{W}}]\in \Omega _4(G).
\end{gather*}

If $W\co M\rightarrow M$ and $W'\co M\rightarrow M$ are cobordant, then $\hat{W}$ and
$\hat{W'}$ are cobordant.
Hence we have a map
\begin{gather*}
  \tr\co \End_\modC (M)\rightarrow \Omega _4(G).
\end{gather*}

For two cobordisms $M\xto{W}M'\xto{{W'}}M$ in $\modC $, we have
the {\em trace identity}
\begin{gather}
  \label{e7}
  \tr({W'}\circ W)  =  \tr(W\circ{W'}).
\end{gather}

\begin{remark}
  The map $\tr\co \End_\modC (M)\rightarrow \Omega _4(G)$ is a group homomorphism.
  We do not need this fact in the rest of this paper.
\end{remark}

\subsection{Functor induced by a $3$--cobordism}
\label{sec:functor-induced-3}
Let $\Sigma =(\Sigma ,\rho _\Sigma )$ and $\Sigma '=(\Sigma ',\rho _{\Sigma '})$ be two $G$--surfaces, and
let $M_0=\tri{M_0}$ be a cobordism between $\Sigma '$ and $\Sigma $, i.e., $\tri{M_0}$
is a $(\Sigma '\sqcup(-\Sigma ))$--bordered $G$--$3$--manifold.  Then we have a functor
\begin{gather*}
  F_{M_0}\co \modC _\Sigma \rightarrow \modC _{\Sigma '}
\end{gather*}
defined as follows.

For an object $M=\triM\in \Ob(\modC _{\Sigma })$, define
\begin{gather*}
  F_{M_0}(\triM)=\tri{F_{M_0}(M)},
\end{gather*}
where
\begin{gather*}
  F_{M_0}(M) = M\cup _\Sigma M_0,\\
  \rho _{F_{M_0}(M)}=\rho _M\cup \rho _{M_0},\\
  \phi _{F_{M_0}(M)}=\phi _{M_0}|_{\Sigma '}.
\end{gather*}
To simplify the notations, we set $M''=M\cup _\Sigma M_0$.

For a morphism $$\tri{W}
\co \triM\rightarrow \tri{M'}$$ in $\modC _\Sigma $, set
\begin{gather*}
  F_{M_0}(\tri{W})=\tri{F_{M_0}(W)}.
\end{gather*}
Here
\begin{gather*}
  F_{M_0}(W) = C_{M''}\cup _M W
\end{gather*}
is obtained by gluing $C_{M''}$ and $W$ along $M$ using the maps
\begin{gather*}
  M \overset{\phi _W|_{M}}{\hookrightarrow} \partial W\quad \text{and}\quad
  M \ct M\times \{0\}\subset M''\times \{0\}\subset \partial C_{M''}.
\end{gather*}
We set
\begin{gather*}
  \rho _{F_{M_0}(W)}=\rho _W\cup \rho _{C_{M''}}\co F_{M_0}(W)\rightarrow K(G,1).
\end{gather*}
The map
\begin{gather*}
  \phi _{F_{M_0}(W)}\co (M\cup _\Sigma M_0)\cup _{\Sigma '}(-(M'\cup _\Sigma M_0))\ct \partial (F_{M_0}(W))
\end{gather*}
is defined in an obvious way.
It is not difficult to check that $F_{M_0}$ is a well-defined functor.

The following proposition means that the functor $F_{M_0}$ preserves
the closure map~$\tr$.

\begin{proposition}
  \label{r8}
  Let $\Sigma $, $\Sigma '$ and $M_0$ be as above.  For an endomorphism
  ${W}\co M\rightarrow M$ in $\modC _\Sigma $, we have
  \begin{gather*}
    \tr(F_{M_0}({W})) = \tr({W})\in \Omega _4(G).
  \end{gather*}
\end{proposition}

\begin{proof}
  Set $W'=F_{M_0}(W)$, and let $\hat{W}'$ be the closed $4$--manifold
  associated to $W'$ as defined in Section \ref{sec:closure-map}.
  Consider the cylinder $X=\hat{W}'\times [0,1]$, which is a $5$--cobordism
  between $\hat{W}'$ and itself.  Let $\sim$ be the equivalence
  relation on $\hat{W}'\times \{0\}\subset \partial X$ by
  \begin{gather*}
    ((x,t),0)\sim((x,t'),0)
  \end{gather*}
  for $x\in M_0$ and $t,t'\in [0,1]$.  The $5$--manifold $X/\sim$ is a
  cobordism between $\hat{W}'$ and $\hat{W}$, on which one can
  construct a structure of a cobordism of closed $G$--$4$--manifolds in a
  natural way.  Hence we have the result.
\end{proof}

\subsection{Restatement of Theorem \ref{t2}}
\label{sec:rest-theor-reft2}

As in Sections \ref{sec2} and \ref{sec3}, let $M$ be a compact,
connected, oriented $3$--manifold with $\partial M\neq\emptyset$.  Let $N$ be
a normal subgroup in $\pi _1M$ and set $G=\pi _1M/N$.  Let $q\co \pi _1M\rightarrow G$
be the projection.

Let $\rho _M\co M\rightarrow K(G,1)$ be the composite of the natural maps
\begin{gather*}
  M\xto{} K(\pi _1M,1) \xto{K(q,1)} K(G,1).
\end{gather*}
Set $\rho _{\partial M}=\rho _M|_{\partial M}\co \partial M\rightarrow K(G,1)$.  Note that
$M=(M,\rho _M,\id_{\partial M})$ is a $(\partial M,\rho _{\partial M})$--bordered $3$--manifold.
In the following, we work in the groupoid $\modC =\modC _{(\partial M,\rho _{\partial M})}$.

Let $L$ be an $N$--link in $M$.  Recall that
\begin{gather*}
  W_L=(M\times [0,1])\cup (\text{$2$--handles attached along $L\times \{1\}$}).
\end{gather*}
By abuse of notation, let $W_L$ denote the $4$--manifold obtained from
$W_L$ by ``reducing $\partial M\times [0,1]$'' by the equivalence relation
$(x,t)\sim(x,t')$, $x\in \partial M$, $t,t\in [0,1]$.  One can identify $W_L$
with
\begin{gather*}
  C_{M}\cup (\text{$2$--handles attached along $L\times \{1\}$}).
\end{gather*}

The map $\rho _M\co M\rightarrow K(G,1)$ extends to
\begin{gather*}
  \rho _{W_L}\co W_L\rightarrow K(G,1),
\end{gather*}
which is unique up to homotopy relative to $M$.  Set
$$\rho _{M_L}=\rho _{W_L}|_{M_L}\co M_L\rightarrow K(G,1).$$ Then $W_L$ is a cobordism
between $(\partial M,\rho _{\partial M})$--bordered $3$--manifolds $(M,\rho _M)$ and
$(M_L,\rho _{M_L})$.

Let $L'$ be another $N$--link in $M$.  If there is a $G$--diffeomorphism
\begin{gather*}
  h\co \tri{M_L}\ct \tri{M_{L'}},
\end{gather*}
then we have
\begin{gather}
  \label{e10}
  \eta _G(M,L,L',h)=\theta _4(\tr(W_{L'}^{-1}\circ C_h\circ W_L)),
\end{gather}
where $C_h\co M_L\rightarrow M_{L'}$ is the mapping cylinder of $h$.

Now, we can restate Theorem \ref{t2} as follows.

\begin{theorem}
  \label{r10}
  Let $M$, $N$, $G$ be as above.  Let $L$ and $L'$ be $N$--links in
  $M$.  Then the following conditions are equivalent.
  \begin{enumerate}
  \item $L$ and
    $L'$ are $\delta (N)$--equivalent.
  \item There is a $G$--diffeomorphism $h\co M_L\rightarrow M_{L'}$ such that
    \begin{gather}
      \label{e21}
      \theta _4(\tr(W_{L'}^{-1}\circ C_h\circ W_L))=0.
    \end{gather}
  \end{enumerate}
\end{theorem}

\section{Characterization of $G$--diffeomorphism}
\label{sec5}

In applications of Theorem \ref{r10}, the homological condition
\eqref{e21} could be an obstruction.  In this section we show that this
condition can be eliminated by introducing new moves on
$N$--links corresponding to elements of the homology group $H_4(G)$.

\subsection{Framed link realization of $\alpha \in H_4(G)$}
\label{sec:fram-link-real}
A framed link $L$ in  a $G$--$3$--manifold $(M,\rho _M)$ is said to be
{\em $G$--trivial} if we have $(\rho _M)_*(N_L)=\{1\}$.  In other
words, $L$ is $G$--trivial if each component of $L$ is mapped by
$\rho_M$ to a null-homotopic loop in $K(G,1)$.

\begin{theorem}
  \label{t3}
  Let $G$ be a group, and let $\alpha \in H_4(G)$.  Then there are
  \begin{itemize}
  \item a $G$--$3$--manifold $(V,\rho _V)$ with $V$ a handlebody,
  \item a $G$--trivial framed link $K$ in $(V,\rho _V)$,
  \item a $G$--diffeomorphism $h_V\co V_K\ct V$
  \end{itemize}
  such that we have
  \begin{gather}
    \label{e20}
    \theta _4(\tr({C_{h_V}} \circ {W^V_K}))=\alpha.
  \end{gather}
  Here the cobordism $W_K^V\co V\to V_K$ is defined by
  \begin{gather*}
    W_K^V=C_V\cup (\text{$2$--handles attached along $K\times\{1\}$}).
  \end{gather*}
\end{theorem}

We call $((V,\rho _V),K,h_V)$ a {\em framed link realization} of $\alpha $.

\begin{proof}
  Since $\theta _4\co \Omega _4(G)\rightarrow H_4(G)$ is surjective, $\alpha \in H_4(G)$ is
  represented by a closed, connected, oriented $G$--$4$--manifold
  $(U,\rho _U)$.  Thus we have $(\rho _U)_*([U])=\alpha $, where $[U]\in H_4(U)$ is
  the fundamental class of $U$.

  Suppose that $\pi _1(U)$ is generated by $r(\ge 0)$ elements.  Let $V$
  denote the $3$--dimensional handlebody of genus $r$.  Take an
  embedding $g\co V\hookrightarrow U$ such that $g_*\co \pi _1V\rightarrow \pi _1U$ is
  surjective.  Set $\rho _V=\rho _Ug\co V\rightarrow K(G,1)$.  Then we have a
  $(\partial V,\rho _{\partial V})$--bordered $3$--manifold
  $\tri{V}$ in an obvious way.

  Let $E$ denote the $4$--manifold obtained from $U$ by cutting along
  the $3$--submanifold $g(V)$.  We regard $E$ as a cobordism from $V$ to itself.  Let
  $\phi _E\co V\cup _{\partial V}(-V)\ct\partial E$ be the boundary parameterization.  Let
  $\rho _E\co E\rightarrow K(G,1)$ be the composite of $\rho _U$ with the canonical
  map $E\rightarrow U$.  Then we have a cobordism
  \begin{gather*}
    E=\tri{E}\co \tri{V}\rightarrow \tri{V},
  \end{gather*}
  which represents an endomorphism ${E}\co V\rightarrow V$ in the category
  $\modC _{\partial V}$.  By construction, we have $\tr(E)=[U,\rho _U]$,
  hence
  \begin{gather}
    \label{e6}
    \theta _4(\tr({E}))=\alpha .
  \end{gather}

  Take a handle decomposition of $E$
  \begin{gather}
    \label{e18}
    E\cong
    C_V\cup (\text{$1$--handles})\cup (\text{$2$--handles})\cup (\text{$3$--handles}),
  \end{gather}
  where $C_V$ is the reduced cylinder of $V$.  We will construct a new
  cobordism $E'\co V\rightarrow V$ cobordant to $E$ over $K(G,1)$ such that $E'$ has a handle
  decomposition with only $2$--handles, by handle-trading as
  follows.

  Suppose that there is a $1$--handle $D^3\times [0,1]$ in the handle
  decomposition \eqref{e18}.  Let $\gamma =\{0\}\times [0,1]\subset D^3\times [0,1]$ be the
  core of the $1$--handle.  Since $g_*$ is surjective, it follows that
  there is a path $\gamma '$ in $V\times \{1\}\subset \partial C_V$ such that $\partial \gamma '=\partial \gamma $
  and the union $\gamma '':=\gamma \cup \gamma '$ is null-homotopic in $E$.  Surgery on
  $E$ along $\gamma ''$ (with any of the at most two possible framings) gives a
  $4$--manifold $E_{\gamma ''}$ cobordant to $E$.  Since $\gamma ''$ is
  null-homotopic in $E$, the map $\rho _E\co E\rightarrow K(G,1)$ extends to a map
  $\rho _{X^E_{\gamma ''}}\co X^E_{\gamma ''}\rightarrow K(G,1)$, where
  \begin{gather*}
    X^E_{\gamma ''}=(C_E\times [0,1])\cup (\text{$2$--handle attached along
      $\gamma ''\times \{1\}$})
  \end{gather*}
  is the cobordism between $E$ and $E_{\gamma ''}$ associated with the
  surgery along $\gamma ''$.  Thus $(E,\rho _E)$ is bordant over $K(G,1)$ to
  $E_{\gamma ''}$.  The manifold $E_{\gamma ''}$ admits a handle decomposition
  with the number of $1$--handles less by $1$ than \eqref{e18}.  By
  induction, we can trade all the $1$--handles, and all the $3$--handles by
  duality, to obtain a desired cobordism $\tri{E'}$ between
  $\tri{V}$ to itself.

  Since the cobordism $E'$ has only $2$--handles, it follows that the
  cobordism $E'$ is equivalent to the composite $C_{h_V}\circ W^V_K$,
  where $K$ is a $G$--trivial framed link in $V$, and $C_{h_V}$ is a
  mapping cylinder of a $G$--diffeomorphism $h_V\co V_K\ct V$.
  It follows that
  \begin{gather*}
    \theta _4(\tr({C_{h_V}}\circ{W^V_K}))
    =\theta _4(\tr({E'}))
    =\theta _4(\tr({E}))
    =\alpha .
  \end{gather*}
\end{proof}

\subsection{Moves on framed links associated to framed link realizations}
\label{sec:-moves-h4g}

As in Section \ref{sec4}, let $\Sigma=(\Sigma,\rho_\Sigma)$ be a
$G$--surface and $M=\triM$ a $\Sigma$--bordered $3$--manifold.  Set
\begin{gather*}
  N=\ker((\rho_M)_*\co\pi_1M\to G).
\end{gather*}
A framed link in $M$ is $G$--trivial if and only if it is an $N$--link.

Let $R=((V,\rho _V),K,h_V)$ be a framed link realization of $\alpha
\in H_4(G)$, and let $L$ be an $N$--link $L$ in $M$.  Suppose that
there is an orientation-preserving embedding $f\co V\hookrightarrow
M\setminus L$ such that $\rho _Mf\simeq\rho _V\co V\rightarrow
K(G,1)$.  Then the framed link $L\cup f(K)$ in $M$ is again an $N$--link.
We say that $L\cup f(K)$ is obtained from $L$ by an {\em $R$--move}.

An $R$--move preserves the $G$--diffeomorphism class of results of surgery.
Indeed, there is a $G$--diffeomorphism
\begin{gather*}
  h\co M_{L\cup f(K)}\ct M_L
\end{gather*}
obtained by gluing $h_V\co V_K\ct V$ and $\id_{M\setminus \opint
f(V)}$.

\begin{proposition}
  \label{r18}
  In the above situation, we have
  \begin{gather*}
    \eta _G(M,L\cup f(K),L,h)
    =\theta _4(\tr( W_{L}^{-1}\circ C_h \circ W_{L\cup f(K)}  ))
    =\alpha \in H_4(G).
  \end{gather*}
\end{proposition}

\begin{proof}
  Note that the cobordism $W_{L\cup f(K)}\co M\rightarrow M_{L\cup
  f(K)}$ is a composition of two cobordisms
  \begin{gather*}
    M \xto{W_L} M_L \xto{W^{M_L}_{f(K)}} M_{L\cup f(K)},
  \end{gather*}
  where
  \begin{gather*}
    W_{f(K)}^{M_L}=C_{M_L}\cup (\text{$2$--handles attached along $f(K)\times \{1\}$})
  \end{gather*}
  and we identify $M_{L\cup f(K)}$ with $(M_L)_{f(K)}$.
  Hence we have
  \begin{gather*}
    \begin{split}
      \theta _4(\tr( W_{L}^{-1}\circ C_h \circ W_{L\cup f(K)}  ))
      =&\theta _4(\tr( W_{L}^{-1}\circ C_h \circ W^{M_L}_{f(K)} \circ W_L  ))\\
      =&\theta _4(\tr( C_h \circ W^{M_L}_{f(K)} \circ W_L \circ W_{L}^{-1} ))
      \quad \quad \text{by \eqref{e7}}
      \\
      =&\theta _4(\tr( C_h \circ W^{M_L}_{f(K)} ))
      \\
      =&\theta _4(\tr( F_{M_L\setminus \opint f(V)}( C_{h_V} \circ W^V_K )))\\
      =&\theta _4(\tr( C_{h_V} \circ W^V_K ))
      \quad \quad \quad \quad\quad \quad  \text{by Prop. \ref{r8}}\\
      =&\alpha .
    \end{split}
  \end{gather*}
\end{proof}

The following result follows immediately from Theorem \ref{r10} and
Proposition \ref{r18}.

\begin{proposition}
  \label{r19}
  Let $R$ and $R'$ be two framed link realizations of $\alpha\in
  H_4(G)$.  Let $L$ be an $N$--link in $M$.  Suppose that we can
  obtain an $N$--link $L_R$ (resp.\ $L_{R'}$) from $L$ by an $R$--move
  (resp.\ $R'$--move).  Then $L_R$ and $L_{R'}$ are
  $\delta(N)$--equivalent.
\end{proposition}

\subsection{$M$--applicable framed link realizations}
\label{sec:m-applicable-framed}

A framed link realization $R=((V,\rho _V),K,h_V)$ of $\alpha \in
H_4(G)$ is said to be {\em $M$--applicable} if we have
\begin{gather}
  \label{e1}
  (\rho_V)_*(\pi_1V)\subset (\rho_M)_*(\pi_1M).
\end{gather}
This condition is equivalent to that one can apply an $R$--move to one
(and in fact every) $N$--link $L$ in $M$.

If $(\rho_M)_*$ is surjective, then every framed link realization of
$\alpha$ is $M$--applicable.

Now we consider the condition for $\alpha\in H_4(G)$ to be realized by
an $M$--applicable framed link realization.  Set
\begin{gather*}
  G_M:= (\rho_M)_*(\pi_1M)\subset G,
\end{gather*}
and let $j\co G_M\to G$ be the inclusion homomorphism.
Let $j_*\co H_4(G_M)\to H_4(G)$ be the induced homomorphism.

\begin{proposition}
  \label{r22}
  Let $\alpha\in H_4(G)$.  Then there is an $M$--applicable framed
  link realization $R$ of $\alpha$ if and only if $\alpha\in
  j_*(H_4(G_M))$.  In particular, if $(\rho_M)_*$ is surjective, then
  for every $\alpha\in H_4(G)$ there is an $M$--applicable framed
  link realization of $\alpha$.
\end{proposition}

\begin{proof}
  Let $R=((V,\rho _V),K,h_V)$ be an $M$--applicable framed link
  realization of $\alpha$.  There is an orientation-preserving
  embedding $f\co V\hookrightarrow M$ such that $\rho_M f\simeq
  \rho_V$.  By Proposition \ref{r18}, we have
  \begin{gather*}
    \alpha
    = \theta_4(\tr(C_h\circ W_{f(K)}))\in j_*(H_4(G_M)),
  \end{gather*}
  where $h=h_V\cup\id_{M\setminus\opint f(V)}\co M_{f(K)}\ct M$, since
  the map $C_h\circ W_{f(K)}\to K(G,1)$ factors through
  $K(G_M,1)$.

  Conversely, suppose $\alpha\in j_*(H_4(G_M))$.  Let
  $\tilde{\alpha}\in H_4(G_M)$ be a lift of $\alpha$ along $j_*$,
  i.e., we have $j_*(\tilde{\alpha})=\alpha$.  By Theorem \ref{t3},
  there is a framed link realization
  $\tilde{R}=((V,\tilde{\rho}_V),K,h_V)$ of $\tilde{\alpha}$.  Then
  $R:=((V,\rho _V),K,h_V)$ is a framed link realization of $\alpha$,
  where $\rho_V=K(j,1)\tilde{\rho}_V$.  (Here $K(j,1)\co K(G_M,1)\to
  K(G,1)$ is the map induced by $j$.)  Clearly, $R$ is
  $M$--applicable.
\end{proof}

By Proposition \ref{r19}, for each $\alpha\in j_*(H_4(G_M))$ the notion of
$R$--move up to $\delta(N)$--equivalence does not depend on the choice
of $R$.  We say that two $N$--links in $M$ are related by an {\em
$\alpha $--move} if they are related by an $R$--move for some
$M$--applicable framed link realization $R$ of $\alpha $.

\subsection{Characterization of $G$--diffeomorphisms}
\label{sec:char-g-diff}

Let $\{\alpha_i\}_{i\in I}$ be an indexed set of generators of
$j_*(H_4(G_M))$.  Theorem \ref{general} below characterizes
$G$--diffeomorphism of results of surgery along $N$--links in $M$.

\begin{theorem}
  \label{general}
  Let $L$ and $L'$ be
  $N$--links in $M$.  Then the following conditions are equivalent.
  \begin{enumerate}
  \item $M_L$ and $M_{L'}$ are $G$--diffeomorphic.
  \item $L$ and $L'$ are related by a sequence of $\alpha _i$--moves
    for $i\in I$ and $\delta (N)$--equivalence.
  \end{enumerate}
\end{theorem}

\begin{proof}
  Clearly, (ii) implies (i).  We prove the reverse implication.

  By assumption, there is a $G$--diffeomorphism $h\co M_L\ct M_{L'}$.
  Set
  \begin{gather*}
    \alpha_{L,L',h}:=\theta_4(\tr(W_{L'}^{-1}\circ C_h \circ W_L))
    \in j_*(H_4(G_M)).
  \end{gather*}
  Since $\{\alpha_i\}_{i\in I}$ generates $j_*(H_4(G_M))$, the element
  $\alpha_{L,L',h}$ can be expressed as a sum of copies of
  $\pm\alpha_i$, $i\in I$.  Thus, by modifying the framed link $L$ by
  the $\alpha_i$--moves, we may assume
  $\alpha_{\tilde{L},L',\tilde{h}}=0$ for the framed link $\tilde{L}$
  obtained from $L$ by these moves and an appropriate diffeomorphism
  $\tilde{h}\co M_{\tilde{L}}\ct M_{L'}$.  Then, by Theorem \ref{r10},
  it follows that $\tilde{L}$ and $L'$ are $\delta(N)$--equivalent.
\end{proof}

We expect that there will be some applications of Theorem
\ref{general} to $3$--dimensional Homotopy Quantum Field Theory (HQFT)
with target space $K(G,1)$ \cite{Turaev:HQFT}.

\section{IHX-moves}
\label{sec6}

In this section, we define $Y_2$--claspers in a $3$--manifold, which
are a special kind of claspers introduced in
\cite{Goussarov1,Goussarov2,H:claspers} and used in the theory of
finite type invariants of links and $3$--manifolds
\cite{Bar-Natan,Ohtsuki:book}.  To each clasper a framed link is
associated on which one can perform surgery.  We define an IHX-move on
the framed links associated to the disjoint union of $Y_2$--claspers.
This move preserves the result of surgery up to diffeomorphism.  An
IHX-move is closely related to the IHX-relation in the theory of
finite type invariants.  This move is related to a handle
decomposition of the $4$--torus $T^4$.

\subsection{$Y_2$--claspers}
\label{sec:y2-claspers}

Let $M$ be a compact, connected, oriented $3$--manifold.
We define $Y_2$--claspers in $M$, which is a special kind of tree
claspers.

A {\em $Y_2$--clasper} in $M$ is a subsurface embedded in the interior
of $M$ which is decomposed into four annuli, two disks and five bands
as depicted in Figure \ref{Y2} (a).  \FI{Y2}{(a) $Y_2$--clasper $T$.
(b) Drawing of $T$.} We usually depict a $Y_2$--clasper as a framed
graph as in Figure \ref{Y2} (b) using the blackboard framing
convention.

We associate to a $Y_2$--clasper $T$ in $M$ a $2$--component framed link
$L_T$ in a small regular neighborhood $N(T)$ of $T$ in $M$ as
depicted in Figure \ref{Y2-assoc} (a).
\FI{Y2-assoc}{(a) The framed link $L_T=L_{T,1}\cup L_{T,2}$ associated to the
$Y_2$--clasper $T$.  (b) Another framed link $L_T^\adm$ associated to
$T$.}  Note that the framed link $L_T$ is $\Z  $--null-homologous in
$N(T)$, hence in $M$.  Surgery along the $Y_2$--clasper $T$ is defined
to be surgery along the associated framed link $L_T$.

Figure \ref{Y2-assoc} (b) shows another framed link $L_T^\adm$
associated to $T$, called the {\em associated admissible framed link}
of $T$, which is used in Section \ref{sec:admissible-ihx-moves-1}.

\begin{lemma}
  \label{r2}
  The framed links $L_T^\adm$ and $L_T$ in $N(T)$ are $\delta $--equivalent.
\end{lemma}

\begin{proof}
  By using one stabilization and two handle-slides, we obtain from
  $L_T$ the framed link $L'_T$ depicted in Figure \ref{Y2-assoc-2}.
  \FI{Y2-assoc-2}{The framed link $L'_T$.}  Then, by handle-sliding
  the middle component over the other two components in $L'_T$, we
  obtain~$L^\adm_T$.
\end{proof}

\subsection{IHX-claspers and IHX-links}
\label{sec:ihx-claspers-ihx}
Let $\gamma =\gamma _1\cup \dots \cup \gamma _4$ be a trivial string link in the cylinder
$D^2\times [0,1]$, i.e., $\gamma $ is a proper $1$--submanifold of $D^2\times [0,1]$
of the form $\text{(four points in $\opint D^2$)}\times [0,1]$.
Let $N(\gamma )\subset D^2\times [0,1]$ be a small tubular neighborhood of $\gamma $ in
$D^2\times [0,1]$, and set
\begin{gather*}
  V_4 = \overline{(D^2\times [0,1])\setminus N(\gamma )}.
\end{gather*}

Let $T_{IHX}=T_1\cup T_2\cup T_3\subset V_4$ be the disjoint union of three
$Y_2$--claspers $T_1,T_2,T_3$ as depicted in Figure \ref{IHX}.
\FI{IHX}{$T_{IHX}\subset V_4$}
We call $T_{IHX}$ the {\em IHX-clasper}.

\begin{theorem}
  \label{t-IHX}
  Surgery along $T_{IHX}$ preserves the manifold $V_4$.  More
  precisely,
  There is a diffeomorphism
  \begin{gather}
    \label{e25}
    h_V\co (V_4)_{T_{IHX}}\ct V_4
  \end{gather}
  restricting to $\id_{\partial V_4}$.  (Note that such a diffeomorphism is
  unique up to isotopy relative to $\partial V_4$.)
\end{theorem}

Theorem \ref{t-IHX} is closely related to the IHX relation in the
theory of finite type invariants.  Similar results, with different
configurations of $Y_2$--claspers, have been obtained in \cite{Goussarov2,CT}.

To prove Theorem \ref{t-IHX}, we need the following.

\begin{lemma}
  \label{l23}
  Let $T_1\subset V_4$ be the first component of $T_{IHX}$, see Figure
  \ref{r23} (a).  By surgery along $T_1$, we obtain from $\gamma $ a
  pure braid $\beta _1:=\gamma _{T_1}$ as depicted in Figure \ref{r23}
  (b).  (Here string links are considered to be framed, and we use the
  blackboard framing convention.)  \FI{r23}{(a) The $Y_2$--clasper
  $T_1$ and the trivial string link $\gamma $.  (b) The pure braid
  $\beta _1$.}
\end{lemma}

\begin{proof}
  By clasper calculus (see \cite{H:claspers}) we can transform
  $(\gamma ,T_1)$ into $(\beta _1,\emptyset)$ as follows.  We use the
  two clasper operations, which do not change the isotopy class of the
  result of surgery, depicted in Figure \ref{m2m10}.
  \FI{m2m10}{Moves 2 and 10 from \cite{H:claspers}}
  First, apply move 2 to the clasper
  $T_1$ as shown in Figure~\ref{T1-braid}~(a).  \FI{T1-braid}{(a)
  $T_1$ after move 2.  (b) $\gamma$ after surgery along $T_1$.}  Then,
  apply move 10 twice to obtain Figure~\ref{T1-braid}~(b), which is
  isotopic to $\beta_1$.
\end{proof}

\begin{proof}[Proof of Theorem \ref{t-IHX}]
  First, we see that the pairs $(\gamma ,T_2)$ and $(\gamma ,T_3)$ are conjugate
  with $(\gamma ,T_1)$ as follows.
  Let $\alpha ^{\pm 1}$ be the braids depicted in Figure \ref{a}.
  \FI{a}{Braids $\alpha$ and $\alpha^{-1}$}
  Then we have
  \begin{gather*}
    (\gamma ,T_2)\cong \alpha ^2(\gamma ,T_1)\alpha ^{-2},\\
    (\gamma ,T_3)\cong \alpha (\gamma ,T_1)\alpha ^{-1}.
  \end{gather*}
  Here the composition $\beta \beta '$ of two tangles $\beta$ and $\beta'$ possibly with claspers
  is obtained by stacking $\beta $ on the top of $\beta '$.

  Then, the result $\gamma _{T_2}$ from $\gamma $ by surgery along $T_2$ is the
  conjugate
  \begin{gather*}
    \gamma _{T_2} \cong \alpha ^2 \gamma _{T_1}  \alpha ^{-2}\cong \alpha ^2\beta _1\alpha ^{-2}=:\beta '_2,
  \end{gather*}
  where we used $\gamma _{T_1}\cong\beta _1$ (Lemma \ref{l23}).
  Similarly, we have
  \begin{gather*}
    \gamma _{T_3} \cong \alpha  \gamma _{T_1}  \alpha ^{-1}\cong \alpha \beta _1\alpha ^{-1}=:\beta '_3.
  \end{gather*}

  These are pure braids depicted below: \FIG{t-IHX-1}.

  By isotopy we obtain the braids $\beta_2$ and $\beta_3$
  \FIG{t-IHX-2} Now, one can check that the composition
  $\beta_1\beta_2\beta_3$ is isotopic to the trivial string link
  $\gamma$.  Thus, surgery along $T_{IHX}$ preserves $V_4$.
\end{proof}

\begin{remark}
  Theorem \ref{t-IHX} may be regarded as a topological version of the
  Witt-Hall identity
  \begin{gather*}
    [z,[y^{-1},x]]^{y^{-1}}\cdot [y,[x^{-1},z]]^{x^{-1}}
    \cdot [x,[z^{-1},y]]^{z^{-1}}=1
  \end{gather*}
  in the free group $\langle x,y,z\rangle$ inside the $4$-strand pure braid group, where $x,y,z$ are defined
  by \FIG{F3-PB} and $[x,y]=xyx^{-1}y^{-1}$ is the commutator.
\end{remark}

\subsection{IHX-moves}
\label{sec:ihx-moves}

Let $L$ be a framed link in a $3$--manifold $M$.  Let
$f\co V_4\hookrightarrow M\setminus L$ be an orientation-preserving embedding.
Then the framed links $L$ and $L\cup f(L_{IHX})$ are said to be related
by an {\em IHX-move}.

An IHX-move preserves the result of surgery.  More precisely, there is
a diffeomorphism
\begin{gather*}
  h\co M_{f(L_{IHX})}\rightarrow M
\end{gather*}
restricting to $\id_{M\setminus \opint{f(V)}}$, which is unique up to isotopy
through such diffeomorphisms.  Indeed, the diffeomorphism $h$ is
obtained by gluing the composite
\begin{gather*}
  f(V)_{f(L_{IHX})}\cong V_{L_{IHX}} \xtoc{h_V}V\cong f(V)
\end{gather*}
and $\id_{M\setminus \opint f(V)}$.

Note that if $L'$ is obtained from a $\Z  $-(resp.\ $\Q$-)null-homologous
framed link $L$ by an IHX-move, then $L'$ is
again $\Z$-(resp.\ $\Q$-)null-homologous.

\section{A handle decomposition of $T^4$}
\label{sec7}

In this section, we construct a new handle decomposition of the
$4$--torus $T^4$ involving the IHX-link.

Consider the framed link with dotted circles obtained from the
IHX-link $L_{IHX}\subset V_4\subset D^2\times [0,1]$ as follows.  We embed $D^2\times [0,1]$
into $S^3$, close the trivial string link $\gamma $ in a natural way to
obtain an unlink $J=J_1\cup \dots \cup J_4$, and put a dot on each component of
$J$.   Here,
each $Y_2$--clasper $T_i$ of $T_{IHX} = T_1\cup T_2\cup T_3$ is
regarded as its associated framed link which we denote by $K_i \cup
K_i'$.  The framed link
\begin{gather*}
  (J_1\cup \dots \cup J_4)\cup (K_1\cup K_1'\cup K_2\cup K_2'\cup K_3\cup K_3') \subset  S^3
\end{gather*}
gives a handlebody $W^{(2)}$ consisting of one $0$--handle
$W^{(0)}=B^4$, four $1$--handles $B_1,\ldots ,B_4$ corresponding to
$J_1,\ldots ,J_4$, and six $2$--handles $H_1,H_1',H_2,H_2',H_3,H_3'$
corresponding to $K_1, K_1', K_2, K_2', K_3, K_3'$.  We set
\begin{gather*}
  W^{(1)}=W^{(0)}\cup B_1\cup B_2\cup B_3 \cup B_4,\\
  W^{(2)}= W^{(1)} \cup H_1\cup H_1'\cup H_2\cup H_2' \cup H_3 \cup H_3'.
\end{gather*}

Since surgery along $K_1\cup K_1'\cup K_2\cup K_2'\cup K_3\cup K_3'$
preserves the result of surgery, we have
\begin{gather*}
  \partial W^{(2)}\cong \partial W^{(1)} \cong \sharp^4(S^2\times S^1).
\end{gather*}
Hence we can attach four $3$--handles and one $4$--handle to obtain an
oriented, closed $4$--manifold $W$.

\begin{theorem}
  \label{IHX-T4}
  The $4$--manifold $W$ is diffeomorphic to the $4$--torus $T^4$.  Thus
  the framed link obtained from Figure \ref{IHX-1} by replacing
  $Y_2$--claspers with the associated framed link presents a handle
  decomposition of $T^4$.
  \FI{IHX-1}{Handle decomposition for the $4$--manifold $W$.  Here,
  each $Y_2$--clasper represents the associated $2$--component framed
  link.}
\end{theorem}

\begin{proof}
  We start from the following handle decomposition of $T^4$ obtained
  by Akbulut in \cite[Figure 4.5]{Akbulut}.  \FIG{T4-akbulut}

  We perform a sequence of handle-slides on the six $2$--handles,
  i.e., on the link $k= k_1\cup k_1'\cup k_2 \cup k_2' \cup k_3 \cup
  k_3'$.

  Slide $k_1'$ twice over $k_3'$ as follows.
  \FIG{BS-1}

  Slide $k_1$ twice over $k_3$ as follows.
  \FIG{BS-2}

  Slide $k_3$ twice over $k_2$ as follows.
  \FIG{BS-3-new}

  Slide $k_2'$ twice over $k_1'$ as follows.
  \FIG{BS-4}

  After isotopy, we obtain the following.
  \FIG{S-1}

  By isotopy, the three $2$--component links can be separated as follows.
  \FIG{S-1b}

  The following shows the result after rearranging the dotted
  circles.
  \FIG{S-1c}

  In clasper calculus this corresponds to the following.
  \FIG{C-T4-2}

  Now, scale down the outermost dotted circle by isotopy passing under
  the second one until it becomes the second circle.  This yields
  Figure \ref{IHX-1}.
\end{proof}

\section{Kirby calculus for $\Q$--null-homologous framed links}
\label{sec8}

Let $M$ be a compact, connected, oriented $3$--manifold, and let
$P=\{p_1,\ldots ,p_t\}\subset \partial M$ be as in Section \ref{sec2}.  In this section,
we consider the case where
\begin{gather*}
  N= \ker(\pi _1M\rightarrow H_1(M,\Q)).
\end{gather*}
The quotient
\begin{gather*}
  G=\pi _1M/N\cong H_1M/(\text{torsion})
\end{gather*}
is a free abelian group.  We fix an identification
\begin{gather*}
  G=\Z  ^r,
\end{gather*}
where $r=\rank H_1M$.

\subsection{The homology group $H_4(\Z  ^r)$}
\label{sec:h4r}

In the following, we often identify $H_1(G)=H_1(\Z^r)$ with $G=\Z^r$.

As is well known, the Pontryagin product (see e.g.\ \cite{Hatcher})
\begin{gather*}
  H_1(\Z^r)\otimes H_1(\Z^r)\otimes H_1(\Z^r)\otimes H_1(\Z^r)\longrightarrow H_4(\Z^r)
\end{gather*}
induces an isomorphism
\begin{gather}
  \label{e23}
  p\co \bigwedge^4H_1(\Z  ^r)\ct H_4(\Z  ^r).
\end{gather}

Define $y_1,\ldots ,y_4\in H_1T^4$ by
\begin{gather*}
  y_1=[S^1\times \pt\times \pt\times \pt],\;\dots ,\;
  y_4=[\pt\times \pt\times \pt\times S^1].
\end{gather*}
Then $y_1,\ldots ,y_4$ generate $H_1T^4\cong\Z  ^4$.
We have
\begin{gather}
  \label{e28}
  p(y_1\wedge\dots \wedge y_4)=[T^4],
\end{gather}
where $[T^4]\in H_4T^4$ is the fundamental class.

The following lemma follows from the definition of the Pontryagin
product.

\begin{lemma}
  \label{r24}
  Let $\rho _{T^4}\co T^4\rightarrow K(\Z  ^r,1)$ be a map.  Then we have
  \begin{gather*}
    (\rho _{T^4})_*([T^4]) = p((\rho _{T^4})_*(y_1)\wedge\dots \wedge(\rho _{T^4})_*(y_4))\in H_4(\Z  ^r).
  \end{gather*}
\end{lemma}

\begin{proof}
  We have
  \begin{gather*}
    \begin{split}
      (\rho _{T^4})_*([T^4])
      &=(\rho _{T^4})_*(p(y_1\wedge\dots \wedge y_4))\quad \quad \quad \quad
      \quad \text{by \eqref{e28}}\\
      &= p((\rho _{T^4})_*(y_1)\wedge\dots \wedge(\rho _{T^4})_*(y_4))\quad \text{by
      naturality of $p$}.
    \end{split}
  \end{gather*}
  Here we used the fact that $\rho _{T^4}\co T^4\rightarrow K(\Z  ^r,1)=T^r$ is
  homotopic to a Lie group homomorphism.
\end{proof}

\subsection{Effect of an IHX-move in $H_4(\Z  ^r)$}
\label{sec:effect-an-ihx}

As in Section \ref{sec6}, let $V=V_4$ be a handlebody of
genus $4$ obtained from the cylinder $D^2\times [0,1]$ by removing the
interiors of the tubular neighborhood of a trivial $4$--component
string link $\gamma =\gamma _1\cup \dots \cup \gamma _4$.  For $i=1,\ldots ,4$, let $x_i\in H_1V$
be the meridian to $\gamma _i$.

Suppose that we are given a $\Z  ^r$--manifold $(M,\rho _M)$ such that
$(\rho _M)_*\co \pi _1M\rightarrow \Z  ^r$ is surjective.

Let $y_1,\ldots ,y_4\in H_1M$.
Let $f\co V\hookrightarrow M$ be an orientation-preserving
embedding such that $f_*(x_i)=y_i$, $i=1,\ldots ,4$.

Set $\rho _V=\rho _M f\co V\rightarrow K(\Z  ^r,1)$.  Then $(V,\rho _V)$ is a $\Z  ^r$--manifold.

Recall that $L_{IHX}$ denotes the IHX-link in $V$.
Set $L=f(L_{IHX})$, which is a $\Z  $--null-homologous framed link in
$M$.  The diffeomorphism $h_V\co V_{L_{IHX}}\ct V$ naturally extends to
a diffeomorphism
\begin{gather*}
  h=h_V\cup \id_{M\setminus \opint f(V)}\co M_{L}\ct M.
\end{gather*}

The following result describes the effect of an IHX-move on the
homology class in $H_4(\Z  ^r)$.

\begin{proposition}
  \label{r14}
  In the above situation, we have
  \begin{gather}
    \label{e27}
    \theta _4(\tr(C_h\circ W_L))
    = \pm p(y_1\wedge\dots \wedge y_4)
    \in H_4(\Z  ^r).
  \end{gather}
\end{proposition}

\begin{proof}
  Let $Y=V\cup _\partial (-V)\cong \sharp^4(S^2\times S^1)$ be the double of $V$, and
  let $i\co V\hookrightarrow Y$ be the inclusion.  The diffeomorphism
  $h_V\co V_L\ct V$ extends to $h_Y=h\cup \id_{-V}\co Y_L\ct Y$.  Set
  $\rho _Y=\rho _V\cup \rho _{-V}\co Y\rightarrow K(\Z  ^r,1)$, where $\rho _{-V}\co -V\rightarrow K(\Z  ^r,1)$
  is the same as $\rho _V$.

  By using Proposition \ref{r8} twice for inclusions $M\supset V\subset Y$, we have
  \begin{gather*}
    \theta _4(\tr(C_h\circ W^M_L))
    =\theta _4(\tr(C_{h_{V}}\circ W^{V}_L))
    =\theta _4(\tr(C_{h_{Y}}\circ W^{Y}_L))
  \end{gather*}
  in $H_4(\Z  ^r)$.

  In the following, we will show that the closed $4$--manifold
  $\tr(C_{h_{Y}}\circ W^{Y}_L)$ is cobordant to $W\cong T^4$ over
  $K(\Z  ^r,1)$, where $W$ is defined in Section \ref{sec7}.

  Consider the cylinder $C_Y:=Y\times [0,1]$ and define a map
  $\rho _{C_Y}\co C_Y\rightarrow K(\Z  ^r,1)$ as the composite
  \begin{gather*}
    C_Y\xto{\text{proj}} Y \xto{\rho _Y} K(\Z  ^r,1).
  \end{gather*}

  The $3$--manifold $Y$ is naturally identified with the boundary of
  the $4$--dimensional handlebody
  \begin{gather*}
    Z:=B^4\cup (\text{four $1$--handles}).
  \end{gather*}
  We regard $Z$ as a cobordism $Z\co \emptyset\rightarrow Y$ (over $K(\Z  ^r,1)$)
  from the empty $3$--manifold $\emptyset$ to $Y$.  The
  orientation-reversal $-Z$ of $Z$ is regarded as a cobordism
  $-Z\co Y\rightarrow \emptyset$.  Then the cobordism $C_Y\co Y\rightarrow Y$ is cobordant
  over $K(\Z  ^r,1)$ to $Z\circ(-Z)\co Y\rightarrow Y$.

  Then we have
  \begin{gather*}
    \begin{split}
    \theta _4(\tr(C_{h_Y}\circ W^Y_L))
    =&\theta _4(\tr(C_{h_Y}\circ W^Y_L\circ C_Y))\\
    =&\theta _4(\tr(C_{h_Y}\circ W^Y_L\circ Z\circ (-Z)))\\
    =&\theta _4(\tr((-Z)\circ C_{h_Y}\circ W^Y_L\circ Z))\\
    =&\theta _4(W,\rho _W).
    \end{split}
  \end{gather*}
  The last identity follows from natural diffeomorphism of closed
  $4$--manifolds
  \begin{gather*}
    g\co (-Z)\circ C_{h_Y}\circ W^Y_L\circ Z\ct W,
  \end{gather*}
  The map $\rho _W\co W\rightarrow K(\Z  ^r,1)$ is the one which extends
  $\rho _Y\co Y\rightarrow K(\Z  ^r,1)$.

  By Theorem \ref{IHX-T4}, we have a  diffeomorphism
  \begin{gather*}
    g'\co W\ct T^4.
  \end{gather*}
  Define $\rho _{T^4}\co T^4\rightarrow K(\Z  ^r,1)$ as the composite
  \begin{gather*}
    T^4\xtoc{(g')^{-1}_*}W\xto{\rho _W} K(\Z  ^r,1).
  \end{gather*}
  Clearly, we have
  \begin{gather*}
    \theta _4(W,\rho _W)=\theta _4(T^4,\rho _{T^4}).
  \end{gather*}

  Let $j\co Y\hookrightarrow W$ be the inclusion map.
  By construction, we see that $(g'gji)_*(x_i)\in H_1T^4$, $i=1,\ldots ,4$, are a set
  of generators of $H_1T^4\cong\Z  ^4$.
  Hence we have
  \begin{gather*}
    \begin{split}
    \theta _4(T^4,\rho _{T^4})
    &= \pm p((\rho _{T^4}g'gji)_*(x_1)\wedge\dots \wedge(\rho _{T^4}g'gji)_*(x_4))\\
    &= \pm p((\rho _{V})_*(x_1)\wedge\dots \wedge(\rho _{V})_*(x_4))\\
    &= \pm p(y_1\wedge\dots \wedge y_4).
    \end{split}
  \end{gather*}

  The identity \eqref{e27} follows from the above identities.
\end{proof}

\begin{theorem}
  \label{r25}
  Let $M$ be a compact, connected, oriented $3$--manifold with
  non-empty boundary.  Let $N$ be a normal subgroup of $\pi _1M$ such
  that $\pi _1M/N\cong \Z  ^r$ with $r\ge 0$.  (Here $r$ may or may not be
  equal to the rank of $H_1M$.)  Let $L$ and $L'$ be two $N$--links.
  Then the following conditions are equivalent.
  \begin{enumerate}
  \item $\tri{M_L}$ and $\tri{M_{L'}}$ are $\Z  ^r$--diffeomorphic.
  \item $L$ and $L'$ are related by a sequence of IHX-moves and
    $\delta (N)$--equivalence.
  \end{enumerate}
\end{theorem}

\begin{proof}
  The result follows from Theorem \ref{general} and Proposition
  \ref{r14} since the set
  \begin{gather*}
    \{p(z_1\wedge\dots \wedge z_4)\in H_4(\Z  ^r)\;|\;  z_1,\ldots ,z_4\in H_1M\}
  \end{gather*}
  generates the group $H_4(\Z  ^r)$.
\end{proof}

\begin{remark}
  \label{r27}
  If $r\le 3$, then we do not need IHX-moves in Theorem
  \ref{r25} since $H_4(\Z  ^r)=0$.
\end{remark}

\subsection{Proof of Theorem \ref{Qnull-homologous} for $M$ with
  non-empty boundary}
\label{sec:null-homol-fram}

Here we consider a special case of Theorem \ref{r25}, where $N$ is the
kernel of the map $\pi _1M\rightarrow H_1M\rightarrow H_1(M;\Q)$.

In the present situation, a framed link $L$ in $M$ is an $N$--link if
and only if it is a  $\Q$--null-homologous framed link as defined in
Section \ref{sec:introduction}.
An $N$--move is the same as a {\em $\Q$--null-homologous $K_3$--move}.

The following result includes Theorem \ref{Qnull-homologous} for $M$
with non-empty boundary.

\begin{theorem}
  \label{r9}
  Let $M$ be a compact, connected, oriented $3$--manifold with
  non-empty boundary with $\rank H_1M=r\ge 0$, which we regard as a
  $(\partial M,\rho _{\partial M})$--bordered $\Z  ^r$--manifold $\tri{M}$.  Let $L$ and
  $L'$ be $\Q$--null-homologous framed links in $M$.  Then the
  following conditions are equivalent.
  \begin{enumerate}
  \item $\tri{M_L}$ and $\tri{M_{L'}}$ are $\Z  ^r$--diffeomorphic.
  \item $L$ and $L'$ are related by a sequence of stabilizations,
  handle-slides, $\Q$--null-homologous $K_3$--moves and IHX-moves.
  \item There is a diffeomorphism $h\co M_L\ct M_{L'}$ restricting
    to the identification map $\partial M_L\cong \partial M_{L'}$ such that the
    following diagram commutes
    \begin{gather}
      \label{e24}
      \xymatrix{
	H_1(M_L,P_L;\Q)
	\ar[rr]^{h_*}_{\cong}
	\ar[dr]_{g_L}
	&&
	H_1(M_{L'},P_{L'};\Q)
	\ar[dl]^{g_L'}
	\\
	&
	H_1(M,P;\Q).
	&
      }
    \end{gather}
  \end{enumerate}
  See Section \ref{sec:introduction} for the definition of $g_L,g_{L'}$.
\end{theorem}

\begin{proof}
  By Theorem \ref{r25}, Conditions (i) and (ii) are equivalent.

  By  Proposition \ref{r21} we see that Condition (i) is equivalent to:
  \begin{itemize}
  \item[(iii')] There is a diffeomorphism $h\co M_L\ct M_{L'}$ restricting
  to $\partial M_L\cong \partial M_{L'}$ such that the following groupoid diagram
  commutes:
    \begin{gather}
      \label{e31}
      \xymatrix{
	\Pi (M_L,P_L)
	\ar[rr]^{h_*}_{\cong}
	\ar[dr]_{q_{L}}
	&&
	\Pi (M_{L'},P_{L'})
	\ar[dl]^{q_{L'}}
	\\
	&
	\Pi (M,P)/N
	&
      }
    \end{gather}
    where $N=\ker(\pi _1M\rightarrow H_1(M;\Q))$, and $q_L$, $q_{L'}$ are as defined
    in \eqref{e8}.
  \end{itemize}
  Then one easily checks that Conditions (iii) and (iii') are equivalent.
\end{proof}

\begin{remark}
  \label{r3}
  By Remark \ref{r27} we do not need the IHX-moves in Theorem \ref{r9}
  when $r=\rank H_1M<4$.  In fact, the case $r<4$ of Theorem \ref{r9}
  with IHX-moves omitted can be proved using \cite[Theorem 2.2]{HW},
  since $H_4(\Z^r)=0$.
\end{remark}

\subsection{Proof of Theorem \ref{Qnull-homologous} for a closed
  $3$--manifold $M$}
\label{sec:proof-theor-refqn}

In the situation of Theorem \ref{Qnull-homologous}, suppose that $M$
is a closed oriented $3$--manifold.  Let
$M_0=M\setminus\operatorname{int}B^3$ be the $3$--manifold obtained
from $M$ by removing the interior of a $3$--ball in $M$.  We have
$\partial M_0=S^2$.

Let $L$ and $L'$ be $\mathbb{Q}$--null-homologous framed links in
$M_0\subset M$.

It is clear that condition (i) in Theorem \ref{Qnull-homologous}
for $L,L'\subset M$ and that for $L,L'\subset M_0$ are equivalent.

It is also easy to see that condition (ii) in Theorem
\ref{Qnull-homologous} for $L,L'\subset M$ and that for $L,L'\subset M_0$ are
equivalent.  Here note that $H_1(M_0,\{p\};\mathbb{Q})$ ($p\in\partial
M_0$) and $H_1(M,\emptyset;\mathbb{Q})$ are naturally isomorphic.

Therefore, Theorem \ref{Qnull-homologous} for $M_0$ implies Theorem
\ref{Qnull-homologous} for $M$.

\section{Admissible framed links in $3$--manifolds with free abelian
  first homology group}
\label{sec9}

As mentioned in the introduction, an {\em admissible} framed link $L$
in a $3$--manifold $M$ is a $\Z$--null-homologous framed link such that
the linking matrix of $L$ is diagonal of diagonal entries $\pm1$.
Surgery along an admissible framed link is called {\em admissible
surgery}.  As observed by Cochran, Gerges and Orr \cite{CGO},
admissible surgery preserves the first homology group, and moreover
the torsion linking pairing and the cohomology rings with arbitrary
coefficients.  Using these algebraic invariants, they gave a
characterization of the equivalence relation on closed oriented
$3$--manifolds generated by admissible surgeries.

In this section, we prove Theorem \ref{admissible}, which may be
regarded as Kirby type calculus for admissible framed links.  In order
to prove Theorem \ref{admissible}, we apply the results for admissible
framed links in $S^3$ developed in \cite{H:kirby} to admissible framed
links in more general $3$--manifolds.

\subsection{Admissible IHX-moves}
\label{sec:admissible-ihx-moves-1}
The definition of an {\em admissible IHX-move} on a framed link is the
same as that of an IHX-move except that we use $L^\adm_{IHX}$ instead
of $L_{IHX}$.  If an admissible IHX-move is applied to an admissible
framed link, then the result is again admissible.

The following lemma immediately follows from Lemma \ref{r2}.

\begin{proposition}
  \label{r1}
  An IHX-move and an admissible IHX-move are equivalent under
  $\delta$--equivalence.  More precisely, an IHX-move can be realized
  by a sequence of an admissible IHX-move and finitely many
  stabilizations and handle-slides, and conversely an admissible
  IHX-move can be realized by a sequence of an IHX-move and finitely
  many stabilizations and handle-slides.
\end{proposition}

\subsection{Reduction of Theorem \ref{admissible}}

By Theorem \ref{Znull-homologous}, we see that Theorem
\ref{admissible} follows from the following result.

\begin{proposition}
  \label{r5}
  Let $L$ and $L'$ be two admissible framed links in a compact,
  connected, oriented $3$--manifold $M$.
  Then the following conditions are equivalent.
  \begin{enumerate}
  \item $L$ and $L'$ are related by a sequence of stabilizations,
    band-slides, pair-moves, admissible IHX-moves, and lantern-moves.
  \item $L$ and $L'$ are related by a sequence of stabilizations,
    handle-slides, $\Z  $--null-homologous $K_3$--moves and IHX-moves.
  \end{enumerate}
\end{proposition}

\begin{proof}[Proof of Proposition \ref{r5}, (i) implies (ii)]
  We have seen that stabilizations, band-slides, pair-moves,
  admissible IHX-moves are realized by a sequence of stabilizations,
  handle-slides, $\Z  $--null-homologous $K_3$--moves and IHX-moves.

  We will show that a lantern-move is realized by a sequence of
  stabilizations, handle-slides and $\Z $--null-homologous
  $K_3$--moves.  Let $K$ and $K'$ be the two framed links in $V_3$
  depicted in Figure \ref{lantern}(a) and (b).
  Figure~\ref{lantern-realization-1} shows a sequence of
  stabilizations, handle-slides and $\Z $--null-homologous
  $K_3$--moves from $K$ to a framed link $\tilde {K}$.  Similarly,
  Figure~\ref{lantern-realization-2} shows a sequence of
  stabilizations, handle-slides and $\Z $--null-homologous
  $K_3$--moves from $K'$ to a framed link $\tilde {K'}$.  The links
  $\tilde{K}$ and $\tilde{K'}$ are isotopic.  Thus, there exists a
  sequence of stabilizations, handle-slides and $\Z $--null-homologous
  $K_3$--moves from $K$ to $K'$.
  \begin{figure}[htp]
    \begin{center}
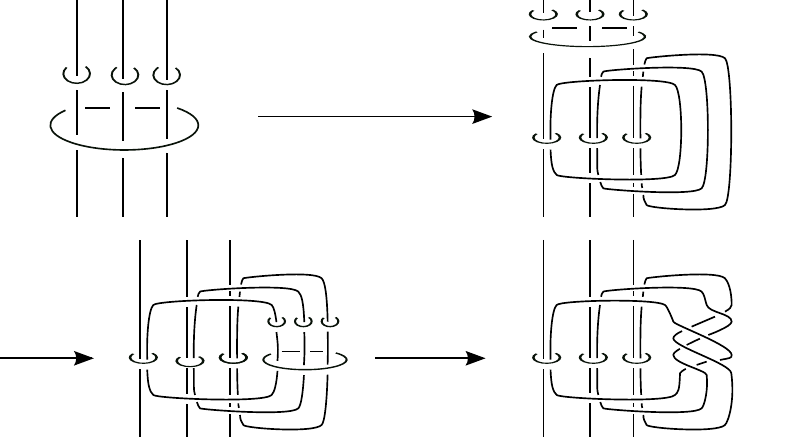
          \end{center}
          \caption{From $K$ to $\tilde{K}$.}
    \label{lantern-realization-1}
\end{figure}
        \begin{figure}[htp]
    \begin{center}
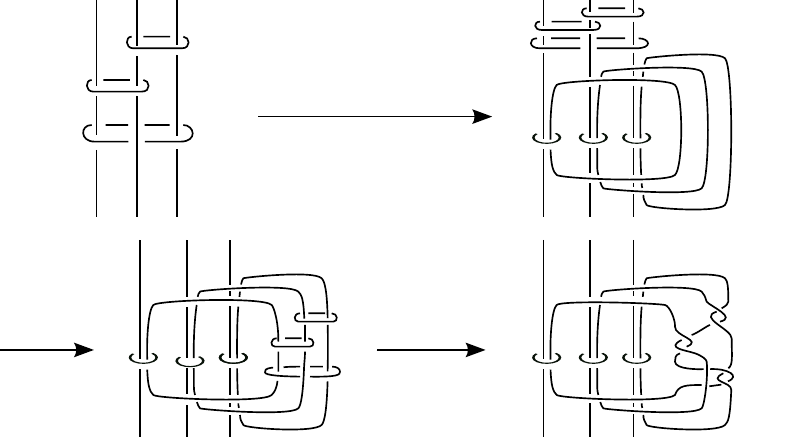
          \end{center}
          \caption{From $K'$ to $\tilde{K'}$.}
    \label{lantern-realization-2}
\end{figure}
\end{proof}

In the rest of this section, we prove that (ii) implies (i).

\subsection{The category $\modS _{M,n}$}
\label{sec:orient-order-admiss}

In the proof that (ii) implies (i) in Proposition \ref{r5}, we use
oriented, ordered framed links.  We briefly recall some definitions
and results from \cite{H:kirby}.

An {\em oriented, ordered framed link} in a $3$--manifold $M$ is a framed
link $L=L_1\cup \dots \cup L_n$ in $M$ such that each component $L_i$ of $L$ is
given an orientation, and the set of components of $L$ is given a
total ordering.  Two oriented, ordered framed links are considered
equivalent if there is an ambient isotopy between them preserving the
orientations and the orderings.

Following \cite{H:kirby}, let $\modL _{M,n}$, $n\ge 0$, denote the set of
equivalence classes of oriented, ordered $n$--component framed links in $M$.
Let $\modE =\modE _n$ denote the set of symbols
\begin{itemize}
\item $\modP _{i,j}$ for $i,j\in \{1,\ldots ,n\}$, $i\neq j$,
\item $\modQ _i$ for $i\in \{1,\ldots ,n\}$,
\item $\modW _{i,j}^\epsilon $ for $i,j\in \{1,\ldots ,n\}$, $i\neq j$, $\epsilon =\pm 1$.
\end{itemize}
For $e\in \modE $, define an $e$--move on $L=L_1\cup \dots \cup L_n\in \modL _{M,n}$ as follows.
\begin{itemize}
\item A {\em $\modP _{i,j}$--move} on $L$ exchanges the order of $L_i$ and
  $L_j$.
\item A {\em $\modQ _i$--move} on $L$ reverses the orientation of $L_i$.
\item A {\em $\modW _{i,j}^\epsilon $--move} on $L$ is a handle-slide of $L_i$
  over $L_j$, see Figure \ref{w-ij}.
  \FI{w-ij}{(a) A $\modW _{i,j}^{+1}$--move.  (b) A $\modW _{i,j}^{-1}$--move}
\end{itemize}
For $L,L'\in \modL _{M,n}$ and $e\in \modL _{M,n}$, we mean by $L\xto{e}L'$ that $L'$ is
obtained from $L$ by an $e$--move.  These moves are called the {\em
  elementary moves}.

Let $\modS _{M,n}$ denote the category such that the objects are the
elements of $\modL _{M,n}$ and the morphisms from $L\in \modL _{M,n}$ to $L'\in \modL _{M,n}$ are the
sequences of elementary moves
\begin{gather*}
  S\co L=L_0\xto{e_1}L_1\xto{e_2}\cdots \xto{e_p}L_p,
\end{gather*}
$p\ge 0$.  The composition of two sequences in $\modS _{M,n}$ is given by
concatenation of sequences, and the identity $1_L$ of $L\in \modL _{M,n}$ is the
sequence of length $0$.

\subsection{The functor $\varphi\co \modS _{M,n}\rightarrow \GL(n;\Z  )$}

There is a functor
\begin{gather*}
  \varphi\co \modS _{M,n}\rightarrow \GL(n;\Z  )
\end{gather*}
from $\modS _{M,n}$ to $\GL(n;\Z  )$, where the group $\GL(n;\Z  )$ of invertible
$n\times n$ matrices with entries in $\Z  $ is regarded as a groupoid with
one object $*$, such that
\begin{align*}
  \varphi(L\xto{\modP _{i,j}}L')&=P_{i,j}:=I_n -E_{i,i}-E_{j,j}+E_{i,j}+E_{j,i},\\
  \varphi(L\xto{\modQ _i}L')&=Q_{i}:=I_n -2E_{i,i},\\
  \varphi(L\xto{\modW _{i,j}^\epsilon }L')&=W_{i,j}^\epsilon :=I_n +E_{i,j},
\end{align*}
where $E_{i,j}=(\delta _{k,i}\delta _{l,j})_{k,l}$.

\begin{lemma}[{\cite[Lemma 2.2]{H:kirby}}]
  \label{hk-lem22}
  If $L,L'\in \modL _{M,n}$ are $\Z  $--null-homologous framed links and if
  $S\co L\rightarrow L'$ is a morphism in $\modS _{M,n}$, then we have the following identity for the linking
  matrices
  \begin{gather*}
    \Lk(L') = \varphi(S)(\Lk(L'))\varphi(S)^t,
  \end{gather*}
  where $(-)^t$ denotes transpose.
\end{lemma}

\begin{theorem}[{\cite[Theorem 2.1]{H:kirby}}]
  \label{hk-thm21}
  If a morphism $S\co L\rightarrow L'$ in $\modS _{M,n}$ satisfies $\varphi(S)=I_n$, then
  $L$ and $L'$ are related by a sequence of band-slides.
\end{theorem}

Note that a band-slide of an oriented, ordered framed link may be
regarded as a morphism in $\modS _{M,n}$ of the form
$L\xto{\modW _{i,j}^{+1}\modW _{i,j}^{-1}}L'$.

\subsection{Reverse sequences}
If
\begin{gather*}
  S\co L=L^0\xto{e_1}L^1\xto{e_2}\cdots \xto{e_k}L^k=L'
\end{gather*}
is a sequence in $\modS _{M,n}$, then there is the {\em reverse sequence}
\begin{gather*}
  \bar{S}\co L'=L^k\xto{\bar{e}_k}\cdots \xto{\bar{e}_2}L^1\xto{\bar{e}_1}L^1=L,
\end{gather*}
where, for $e\in \modE $, we define $\bar{e}\in \modE $ by
\begin{gather*}
  \bar{\modP }_{i,j}=\modP _{i,j},\quad \bar{\modQ }_i=\modQ _i,\quad
  (\bar{\modW }_{i,j}^{\epsilon })=\modW _{i,j}^{-\epsilon }.
\end{gather*}
We have
\begin{gather*}
  \varphi(\bar{S})=\varphi(S)^{-1}.
\end{gather*}

\subsection{Admissible framed links}
\label{sec:admiss-fram-links-1}

An {\em oriented, ordered admissible framed link} in $M$ of type
$(p,q)$, $p,q\ge 0$, is an oriented, ordered, $\Z  $--null-homologous
framed link
\begin{gather*}
  L=L_1\cup \dots \cup L_p\cup L_{p+1}\cup \dots \cup L_{p+q}\subset M
\end{gather*}
such that the linking matrix $\Lk(L)$ of $L$ satisfies
\begin{gather*}
  \Lk(L)=I_{p,q}:=I_p\oplus(-I_q),
\end{gather*}
where $I_p$ denotes the identity matrix of size $p$, and $\oplus$
denotes block sum.

For $p,q\ge 0$, let $\Lapq$ denote the subset of $\modL _{M,p+q}$ consisting
of the equivalence classes of oriented, ordered admissible framed
links in $M$ of type $(p,q)$.  Let $\modS ^{\adm}_{M;p,q}$ denote the full subcategory
of $\modS _{M,p+q}$ such that $\Ob(\modS ^{\adm}_{M;p,q})=\Lapq$.

Let $L,L'\in \Lapq$, and suppose that there is a morphism $S\co L\rightarrow L'$ in
$\modS _{M,p+q}$, i.e., a sequence of elementary moves from $L$ to $L'$.
By Lemma \ref{hk-lem22}, it follows that
\begin{gather}
  I_{p,q}= \varphi(S)I_{p,q}\varphi(S)^t,
\end{gather}
hence
\begin{gather*}
  \varphi(S)\in \Opq :=\{T\in \GL(p+q;\Z  )\;|\; TI_{p,q}T^t=I_{p,q}\}.
\end{gather*}

We use the following result.

\begin{lemma}[{Wall \cite[1.8]{Wall}}]
  \label{hk-lem43}
  If $p,q\ge 2$, then $\Opq$ is generated by the matrices
  \begin{gather}
    \label{Opq-gen}
    \begin{split}
      &P_{i,j}\quad \text{for $i,j\in \{1,\ldots ,p\}$, $i\neq j$},\\
      &P_{i,j}\quad \text{for $i,j\in \{p+1,\ldots ,p+q\}$, $i\neq j$},\\
      &Q_{i}\quad \text{for $i\in \{1,\ldots ,p+q\}$},\\
      &D_{p,q}
      =\left( \begin{matrix} 1&1&0&-1& 0&0\\-1&1&0&0&1&0\\
	0&0&I_{p-2}&0&0&0\\-1&0&0&1&1&0\\0&1&0&-1&1&0\\0&0&0&0&0&I_{q-2}
      \end{matrix}\right).
    \end{split}
  \end{gather}
\end{lemma}

\subsection{$\modD ^{\pm 1}$--moves}
We consider a sequence of elementary moves on oriented, ordered
admissible framed links whose associated matrix is $D_{p,q}$.  The
matrix
\begin{gather*}
  D_{2,2}=
  \left( \begin{matrix} -1&1&-1&0\\-1&1&0&1\\-1&0&1&1\\0&1&-1&1\end{matrix}\right)
  \in  \mathrm{O}(2,2;\Z  )
\end{gather*}
 is a product of the $W_{i,j}^{\pm 1}$ matrices as
\begin{gather*}
  D_{2,2} = W_{2,1}^{-1}W_{3,1}^{-1}W_{2,4}W_{3,4}
  W_{4,3}^{-1}W_{1,3}^{-1}W_{4,2}W_{1,2}.
\end{gather*}

Consider the $4$--component framed links $l,l',\tl\in \modL _{V_4,4}$ in the
handlebody $V_4$ of genus $4$ depicted in Figure \ref{use-hs}.
\FI{use-hs}{} The handlebody $V_4$ is realized as the complement of
the trivial $4$--component string link in the cylinder $D^2\times [0,1]$.
By applying $\modW _{1,2}$, $\modW _{4,2}$, $\modW ^{-1}_{1,3}$, $\modW ^{-1}_{4,3}$
moves to $l$, we obtain $\tl$.  Similarly, by applying $\modW _{2,1}$,
$\modW _{3,1}$, $\modW ^{-1}_{2,4}$, $\modW ^{-1}_{3,4}$ moves to $l'$, we obtain
$\tl$.  Thus we have a sequence $\modD _{2,2}$ from $l$ to $l'$ such that
$\varphi(\modD _{2,2})=D_{2,2}$.

Let $L\in \Lapq$ with $p,q\ge 2$.  Then we can find an
orientation-preserving embedding
\begin{gather*}
  f\co V_4\hookrightarrow M
\end{gather*}
such that $f(l)=L_1\cup L_2\cup L_{p+1}\cup L_{p+2}$ as follows.

By adding three edges $c_{1,j}$, $j=1,2,3$, to $l$ in an appropriate
way, we obtain a $1$--subcomplex $X=l\cup c_{1,2}\cup c_{1,3}\cup c_{1,4}$ of
$V_4$, which is a strong deformation retract of $V_4$.  Take an
embedding $f_X\co X\hookrightarrow M$ such that
$f_X(l)=L_1\cup L_2\cup L_{p+1}\cup L_{p+2}$.  Then $f_X$ extends to an
embedding $f\co V_4\hookrightarrow M$ with the desired
property.

Set $L'=L'_1\cup \dots \cup L'_{p+q}$, where
\begin{gather*}
  L'_1=f(l'_1),\quad
  L'_2=f(l'_2),\quad
  L'_{p+1}=f(l'_3),\quad
  L'_{p+2}=f(l'_4),\\
  L'_i=L_i\quad \text{for $i\in \{1,\ldots,p+q\}\setminus\{1,2,p+1,p+2\}$}.
\end{gather*}
Then there is a sequence $\modD \co L\rightarrow L'$, corresponding to the sequence
$\modD _{2,2}$, such that $\varphi(\modD )=D_{p,q}$.
Similarly, given $L\in \Lapq$, there is a sequence $\modD ^{-1}\co L\rightarrow L'$ such
that $\varphi(\modD ^{-1})=D_{p,q}^{-1}$.  In these situations, $L$ and
$L'$ are said to be related by a {\em $\modD ^{\pm 1}$--move}.

Now, we can prove the following.

\begin{proposition}
  \label{rr11}
  Let $L,L'\in \Lapq$ with $p,q\ge 2$.  Suppose that there is a morphism
  $S\co L\rightarrow L'$ in $\modS _{M,p+q}$, i.e., a sequence of elementary moves
  from $L$ to $L'$.  Then $L$ and $L'$ are related by a sequence of
  \begin{itemize}
  \item band-slides,
  \item $\modP _{i,j}$--moves for $i,j\in \{1,\ldots ,p\},i\neq j$ and for $i,j\in \{p+1,\ldots ,p+q\},i\neq j$,
  \item $\modQ _{i}$--moves for $i\in \{1,\ldots ,p+q\}$,
  \item $\modD ^{\pm 1}$--moves.
  \end{itemize}
\end{proposition}

\begin{proof}
  Express $\varphi(S)\in \Opq$ as
  \begin{gather*}
    \varphi(S) = x_k\cdots x_2x_1,\quad k\ge 0,
  \end{gather*}
  where each $x_i$ is one of the generators given in \eqref{Opq-gen} or
  its inverse.  We can construct a sequence
  \begin{gather*}
    T\co  L=L_0 \xto{x_1}L_1\xto{x_2}\cdots \xto{x_k}L_k=L''
  \end{gather*}
  such that $L_0,\ldots ,L_k\in \Lapq$, and for each $m=1,\ldots ,k$, $L_m$ is
  obtained from $L_{m-1}$ by either a $\modP _{i,j}$--move, a $\modQ _i$--move
  or a $\modD ^{\pm 1}$--move corresponding to $x_m$.  We may regard $T$ as a
  sequence from $L$ to $L''$ of elementary moves, i.e., a morphism from
  $L$ to $L''$ in $\modS _{M,p+q}$, by replacing each $\modD ^{\pm 1}$ move in
  $T$ with the corresponding sequence of eight $\modW _{i,j}^{\pm 1}$--moves.
  Thus, $L$ and $L''$ are related by a sequence of moves listed in the
  proposition (without band-slides).

  Now, since the composite sequence $T \bar{S}\co L'\rightarrow L''$
  satisfies 
  \begin{gather*}
    \varphi(T\bar{S})=\varphi(T)\varphi(S)^{-1}=I_{p+q},
  \end{gather*}
  it follows from Theorem \ref{hk-thm21} that there is a sequence of
  band-slides from $L'$ to $L''$.

  Hence it follows that there is a sequence from $L$ to $L'$ of moves
  listed in the proposition.
\end{proof}

\subsection{Proof of Proposition \ref{r5}, (ii) implies (i)}

Throughout this section, $M$ is a compact, connected, oriented
$3$--manifold such that $H_1M\cong\Z  ^r$ is free abelian.  Let $L$
and $L'$ be two admissible framed links in $M$.  Let
\begin{gather}
  \label{e3}
  S\co  L=L^0\rightarrow L^1\rightarrow \cdots \rightarrow L^k=L'
\end{gather}
be a sequence of $\Z  $--null-homologous framed links between $L$ and
$L'$ such that, for each $i=1,\ldots ,k$, $L^i$ is obtained from $L^{i-1}$
by either stabilization, handle-slide, $\Z  $--null-homologous
$K_3$--move or IHX-move.

\subsubsection{Eliminating IHX-moves}

Note that, for each $i=1,\ldots,k$, there is a diffeomorphism $h_i\co
M_{L^{i-1}}\ct M_{L^i}$ which restricts to the identity outside a
handlebody in $M_{L^{i-1}}$ supporting the relevant move.
Let $h_S\co M_L\ct M_{L'}$ be the diffeomorphism obtained by composing
the diffeomorphism $h_1,\dots,h_k$.

Set
\begin{gather*}
  \eta (S) :=\theta _4(\tr((W^M_{L'})^{-1}\circ C_{h_S}\circ W^M_L))\in H_4(H_1M).
\end{gather*}

Since $H_1M\cong\Z ^r$, the homology group $H_4(H_1M)\cong\bigwedge^4H_1M$ is finitely generated by elements of the form
$p(x_1\wedge\dots \wedge x_4)$ with $x_1,\ldots ,x_4\in H_1M$.  Hence, using
Proposition \ref{r14}, we can construct a sequence $T$ of admissible
IHX-moves
\begin{gather*}
  T\co L'=K^0\rightarrow K^1\rightarrow \cdots \rightarrow K^m=L''
\end{gather*}
from $L'$ to an admissible framed link $L''$
such that
\begin{itemize}
\item $\eta (T)=-\eta (S)$,
\item there are orientation-preserving embeddings $f_1,\ldots ,f_m\co V_4\hookrightarrow
  M\setminus L'$ with mutually disjoint images satisfying
  \begin{gather*}
    K^i=L'\cup f_1(L^\adm_{IHX})\cup \dots \cup f_i(L^\adm_{IHX})
  \end{gather*}
  for $i=0,\ldots ,m$.
\end{itemize}

Then
\begin{gather*}
  \begin{split}
    \eta (TS)
    &=\theta _4(\tr((W^M_{L''})^{-1}\circ C_{h_{TS}}\circ W^M_L))\\
    &=\theta _4(\tr((W^M_{L''})^{-1}\circ C_{h_{T}}\circ C_{h_{S}}\circ W^M_L))\\
    &=\theta _4(\tr((W^M_{L''})^{-1}\circ C_{h_{T}}\circ W^M_{L'}\circ(W^M_{L'})^{-1}\circ C_{h_{S}}\circ W^M_L))\\
    &=\theta _4(\tr((W^M_{L''})^{-1}\circ C_{h_{T}}\circ W^M_{L'}))+\theta _4(\tr((W^M_{L'})^{-1}\circ C_{h_{S}}\circ W^M_L))\\
    &=\eta (T)+\eta (S)=0,
  \end{split}
\end{gather*}
where $h_{TS}\co M_{L}\rightarrow M_{L''}$ is the diffeomorphism associated to the
composite sequence $TS\co L\rightarrow L''$.  Then, by Theorem \ref{t2} with
$N=[\pi _1M,\pi _1M]$, it follows that $L$ and $L''$ are related by a
sequence of stabilizations, handle-slides, and $K_3([\pi _1M,\pi _1M])$--moves,
i.e., $\Z  $--null-homologous $K_3$--moves.

Thus, we may assume without loss of generality that there are no
IHX-moves in the sequence $S$.

\subsubsection{Eliminating $\Z  $--null-homologous $K_3$--moves}
\label{sec:step-2.-eliminating}

Unlike in the other part of this section, here, let us distinguish
ambient isotopic framed links.
We have a sequence $S$ in \eqref{e3} of stabilizations, handle-slides,
$\Z  $--null-homologous $K_3$--moves and ambient isotopies.
By modifying this sequence using ambient isotopy if necessary, we may
assume that there is a handlebody $V$ in the interior of $M$
satisfying the following conditions.
\begin{itemize}
\item All the framed links involved in $S$ are contained in  $V$.
\item For each $i=1,\dots,k$, $L_{i}$ is obtained from $L_{i-1}$ by
  either a stabilization, a handle-slide, a $\Z $--null-homologous
  $K_3$--move or an ambient isotopy in $M$.  (Here and below, a
  ``$\Z$--null-homologous $K_3$--move'' involves framed links that are
  $\Z$--null-homologous in $M$, but not necessarily so in $V$.)
\end{itemize}
Indeed, if we take a handle decomposition of $M$ based on the cylinder
$\partial M\times[0,1]$
\begin{gather*}
  M=(\partial M \times[0,1]) \cup(\text{$1$--handles})\cup(\text{a handlebody $V$}),
\end{gather*}
then any framed link in $M$ can be
isotoped into $V$.  Moreover, we can isotope into $V$ the
handlebodies in which each of the stabilizations, handle-slide,
$\Z $--null-homologous $K_3$--move takes place.

The inclusion $M\setminus V\subset M$ induces surjective
homomorphisms $\pi_1(M\setminus V)\twoheadrightarrow\pi_1M$, and
$[\pi _1(M\setminus V),\pi _1(M\setminus V)]\twoheadrightarrow[\pi _1M,\pi _1M]$.

Since $\pi_1M$ is finitely generated, the commutator subgroup
$[\pi_1M,\pi_1M]$ is generated by the conjugates in $\pi_1M$ of
finitely many elements $x_1,\ldots,x_t \in [\pi_1M,\pi_1M]$, $t\ge0$.
Let $\tx_j\in [\pi _1(M\setminus V),\pi _1(M\setminus V)]$ be a lift
of $x_j$.  We can find an admissible framed link
\begin{gather*}
  K=K^+_1\cup K^-_1\cup \dots \cup K^+_t\cup K^-_t
\end{gather*}
in $M\setminus V$
satisfying the following conditions.
\begin{itemize}
\item[(1)] The (free) homotopy classes of $K^+_j$ and $K^-_j$ are $\tx_j$.
\item[(2)] There are $t$ mutually disjoint annuli $A_1,\ldots ,A_t$ in
  $M\setminus V$ such that $\partial A_j=K^+_j\cup K^-_j$,
\item[(3)] The framing of $K^{\pm}_j$ is $\pm1$.
\end{itemize}

Set $\tL=L\cup K$, $\tL'=L'\cup K$ and $\tL^i=L^i\cup K$,
$i=0,\ldots,k$.  Then $\tL$ (resp.\ $\tL'$) is obtained from $L$
(resp.\ $L'$) by $t$ pair-moves.  Thus, it suffices to show that for
each $i=1,\ldots ,k$, $\tL^{i-1}$ and $\tL^i$ are related by a
sequence of stabilizations, handle-slides and ambient isotopies in
$M$.  We may safely assume that $k=1$.

If $L(=L^0)$ and $L'(=L^1)$ are related by either a stabilization or a
handle-slide in $V$, then clearly $\tL$ and $\tL'$ are related by a
stabilization or a handle-slide.

If $L$ and $L'$ are related by a $\Z$--null-homologous $K_3$--move in $V$, then
let us assume that $L'=L\cup J\cup J'$ is obtained from $L'$ by adding
a $\Z$--null-homologous component $J$ and a small $0$--framed meridian
$J'$ of $J$.  (Of course, the case of $\Z$--null-homologous $K_3$--move
in the other direction is similar.)  Since the homotopy classes of the
$K^\pm_j$ generate $[\pi_1M,\pi_1M]$ normally in $\pi_1M$, it follows
that we can handle-slide $J$ over the $K^\pm_j$ finitely many times to make $J$
null-homotopic in $M$.  Then there is a sequence from $J$ to an
unknot of crossing changes of $J$ with any components of the framed
link other than $J'$.  Such crossing changes can be realized by
handle-slides of link components over $J'$.  Thus we may assume that
$J\cup J'$ is a Hopf link such that $J'$ is of framing $0$ or $+1$.  It is
well known that $J\cup J'$ is related to the empty link by a sequence
of stabilizations and handle-slides.  Hence, $L$ and
$L'$ are related by a sequence of stabilizations and handle-slides in $M$.

If $L$ and $L'$ are ambient isotopic in $M$, then they are related by
a sequence of
\begin{itemize}
\item ambient isotopies in $M\setminus (A_1\cup \dots \cup A_t)$,
\item crossing changes of a component with some $A_j$.
\end{itemize}
We may assume without loss of generality that $L$ and $L'$ are related
by one of these moves.  If $L$ and $L'$ are ambient isotopic in
$M\setminus (A_1\cup \dots \cup A_t)$, then $\tL$ and $\tL'$ are ambient isotopic in
$M$.  If $L$ and $L'$ are related by a crossing change of a component
$L_c$ of $L$ with $A_j$, then $\tL$ and $\tL'$ are related by two
handle-slides.  (Here, we first slide $L_c$ over $K^+_j$, and then we
slide it over $K^-_j$.)

\subsubsection{Eliminating stabilizations}

Now, $L$ and $L'$ are related by a sequence $S$ in \eqref{e3} of
stabilizations and handle-slides.

It is well known that we can exchange the order of consecutive
stabilizations and handle-slides to obtain a new sequence
\begin{gather*}
  S'\co  L \rightarrow \cdots \rightarrow \tL\rightarrow \cdots \rightarrow \tL'\rightarrow \cdots \rightarrow L',
\end{gather*}
where $\tL$ is obtained from $L$ by adding isolated $\pm 1$--framed
unknots by stabilizations, and $\tL'$ is obtained from $\tL$ by a
sequence of handle-slides, and $L'$ is obtained from $\tL'$ by
removing isolated $\pm 1$--framed unknots by stabilizations.
Note that $\tL$ and $\tL'$ are admissible.

We may assume that $\tL$ (and hence $\tL'$) is admissible of type
$(p,q)$ with $p,q\ge 2$, since otherwise we can add the number of
components by using stabilizations.

Thus, we have only to consider the sequence $\tL\rightarrow \cdots \rightarrow \tL'$ of
handle-slides, where $\tL$ and $\tL'$ are admissible of type $(p,q)$
with $p,q\ge 2$.

\subsubsection{Reduction to $\modD ^{\pm 1}$--moves}

Suppose that $L$ and $L'$ are related by a sequence of handle-slides
and that $L$ and $L'$ are admissible of type $(p,q)$ with $p,q\ge 2$.

We fix an orientation and ordering of $L$ and $L'$.  Then $L$ and $L'$
are, as oriented, ordered framed links, related by elementary moves as
defined in Section \ref{sec:orient-order-admiss}.  Then by Proposition
\ref{rr11} it follows that $L$ and $L'$ are related by a sequence of
moves listed in Proposition \ref{rr11}.  Hence $L$ and $L'$, as
non-ordered, non-oriented framed links, are related by a sequence of
band-slides and $\modD ^{\pm 1}$--moves, where each $\modD ^{\pm 1}$--move can be
applied to any $4$--component sublinks of framings $+1,+1,-1,-1$ by
assuming any orientation.

\subsubsection{$\modD ^{\pm 1}$--moves and lantern-moves}

Now it suffices to prove the following lemma.

\begin{lemma}
  \label{r4}
  Suppose that admissible framed links $L$ and $L'$ are related by a
  $\modD ^{\pm1}$--move.
  Then there is a sequence between $L$ and $L'$ of two lantern-moves,
  and finitely many pair-moves.
\end{lemma}

\begin{proof}
  Suppose that $L'$ is obtained from $L$ by a $\modD ^{+1}$--move.

  Let $l=l_1\cup l_2\cup l_3\cup l_4\subset L$ be the sublink of $L$
  involved in the $\modD ^{+1}$--move.  In Figure
  \ref{use-lantern-new}(a), the framed link $l$ in a genus $4$
  handlebody $V_4$ in $M$ is depicted.  Here, as usual, $V_4$ is
  identified with the complement of a trivial string link $\gamma $ in
  the cylinder $D^2\times [0,1]$.  Recall that the meridian to each
  strand of $\gamma $ is null-homologous in $M$, of zero framing,
  and having zero linking number with each component of $L$.

  Then $L'$ is obtained from $L$ by replacing $l$ with a
  $4$--component sublink $l'=l'_1\cup l'_2\cup l'_3\cup l'_4$ in Figure
  \ref{use-lantern-new}(e).
  \FI{use-lantern-new}{}
  We can go from (a) to (e) by using pair-moves and lantern-moves
  as follows.

  Starting at (a), we first apply pair-moves along the meridians to
  the first, second and fourth strands, and then apply a lantern-move
  involving $l_1$ to obtain (b).
  Next, we arrive at (c) by applying a pair-move to remove two
  components which links with the second and fourth strands.

  Similarly, we can go from (e) to (c).  We get from (e) to (d) by
  using pair-moves along the meridians to the first, third and fourth
  strands, and a lantern-move involving $l'_2$.  Then we get at (c) by
  one pair-move.
\end{proof}

\begin{remark}
  \label{r6}
  A lantern-move can be realized by stabilizations, pair-moves, and
  one $\modD ^{\pm 1}$--move.  To see this, one embeds $V_4$ in $M$ in such a
  way that the meridian to the third strand of $\gamma $ is mapped to a
  $0$--framed unknot bounding a disk which does not intersect the other
  components of the framed link.  This amounts to removing the third
  strand of $\gamma $ in the definition of $\modD ^{\pm 1}$--move.
  Then it is not difficult to see that this special
  $\modD^{\pm1}$--move is equivalent to a lantern-move up to
  stabilizations and pair-moves.
\end{remark}

\subsection{Realizing moves with the other moves}
\label{sec:realizing-moves-with}

Here we give a few remarks about realizing some moves in Theorem
\ref{admissible} with the other moves.

In Theorem \ref{admissible}, the pair-moves are not necessary.
Pair-moves on an admissible framed link in $M$ are used to modify the
normal subgroup $N_L$ of $\pi_1M$.  This can be done also by using a
sequence of IHX-moves.  Note that any commutator $[a,b]$ in $\pi_1M$
can be realized as the homotopy class of a component in an IHX-link
$f(L_{IHX})\subset M$.  Here we need to embed $V_4$ in such a way that
\begin{gather*}
  f_*(x_1)=a,\quad  f_*(x_2)=b,\quad f_*(x_3)=f_*(x_4)=1,
\end{gather*}
where $x_1,\ldots,x_4\in \pi_1 V_4$ are as defined in Section
\ref{sec:effect-an-ihx}.

If $\rank H_1M<4$, then the admissible IHX-moves are not necessary in
Theorem \ref{admissible} since $H_4(H_1M)=0$.

\bibliographystyle{plain}

\end{document}